\documentclass{amsart}
\usepackage[utf8]{inputenc}
\usepackage[T1]{fontenc}
\usepackage{amsfonts}
\usepackage{amsmath}
\usepackage{amssymb}
\usepackage{amsthm}
\usepackage{a4wide}
\usepackage{graphicx}
\usepackage{mathrsfs}
\usepackage{dcpic,pictex} 
\usepackage{hyperref} 
\usepackage{cases}

\usepackage{enumerate}

\DeclareMathAlphabet{\mathpzc}{OT1}{pzc}{m}{it}

\newtheorem{theorem}{Theorem}
\newtheorem{proposition}[theorem]{Proposition}
\newtheorem{corollary}[theorem]{Corollary}
\newtheorem{lemma}[theorem]{Lemma}
\theoremstyle{definition}
\newtheorem{definition}{Definition}
\newtheorem{example}{Example}
\newtheorem{remark}{Remark}

\title{Twistor interpretation of slice regular functions}

\author{Amedeo Altavilla}

\address{Amedeo Altavilla, Dipartimento di Ingegneria Industriale e Scienze Matematiche 
Universit\`a Politecnica delle Marche, Via Brecce Bianche,
I-60131 Ancona, Italy, amedeoaltavilla@gmail.com}
\subjclass[2010]{ 53C28, 30G35, 53C55, 14J26 }
\keywords{Twistor spaces, Slice regular functions, Functions of hypercomplex variables, Rational and ruled surfaces.}

\begin{document}

\begin{abstract}
Given a slice regular function $f:\Omega\subset\mathbb{H}\to \mathbb{H}$, with $\Omega\cap\mathbb{R}\neq \emptyset$, it is possible to lift it to 
a surface in the twistor space $\mathbb{CP}^{3}$ of $\mathbb{S}^4\simeq \mathbb{H}\cup \{\infty\}$ (see~\cite{gensalsto}).
In this paper we show that the same result is true if one removes the hypothesis $\Omega\cap\mathbb{R}\neq \emptyset$ on the domain of the function $f$. Moreover we find that if a surface $\mathcal{S}\subset\mathbb{CP}^{3}$ contains
the image of the twistor lift of a slice regular function, then $\mathcal{S}$ has to be ruled by lines.
Starting from these results we find all the projective classes of algebraic surfaces up to degree 3 in $\mathbb{CP}^{3}$
that contain the lift of a slice regular function. In addition we extend and further explore 
the so-called twistor transform, that is a curve in $\mathbb{G}r_2(\mathbb{C}^4)$ which, given a slice regular function, returns the arrangement of lines whose lift carries on. With the explicit expression of the twistor lift and of the twistor transform of a slice regular function we exhibit the set of slice regular functions whose
twistor transform describes a rational line inside $\mathbb{G}r_2(\mathbb{C}^4)$, showing
the role of slice regular functions not defined on $\mathbb{R}$.
At the end we study the twistor lift of a particular slice regular function not defined over the reals. This example
shows the effectiveness of our approach and opens some questions.
\end{abstract}

\maketitle






\section{Introduction}
This paper is devoted to further investigating the relation between orthogonal complex structures
 on subdomains of $\mathbb{R}^4$ and the recent theory of quaternionic 
 slice regular functions. 
 
Given a $2n$-dimensional oriented Riemannian manifold $(\Omega,g)$, an
\textit{almost complex structure}  over $\Omega$ is an endomorphism 
$J:T\Omega\rightarrow T\Omega$, defined over the tangent bundle, 
such that $J^{2}=-id$. An almost complex structure is said to be a \textit{complex structure} if $J$ is \textit{integrable}, 
meaning, for instance, that the associated \textit{Nijenhuis} tensor,
\begin{equation*}
 N_J(X,Y)=[X,Y]+J[JX,Y]+J[X,JY]-[JX,JY],
\end{equation*}
vanishes everywhere for each couple of tangent vectors $X$ and $Y$; it is said to be \textit{orthogonal} if it preserves the
Riemannian product, i.e. $g(JX,JY)=g(X,Y)$ for each couple of tangent vectors $X$ and $Y$
and preserves the orientation of $\Omega$. Collecting everything, an \textit{orthogonal complex
structure} (OCS) is an almost complex structure which is integrable and orthogonal.


The condition for $J$ to be an OCS depends only on the conformal 
class of $g$, so, if $\Omega$ is a four dimensional open subset of $\mathbb{R}^4$
endowed with the Euclidean metric $g_{Eucl}$, then
 the resulting theory is  invariant under the group $SO(5,1)$ of
conformal automorphisms of $\mathbb{R}^4\cup\{\infty\}\simeq\mathbb{H}\cup\{\infty\}\simeq\mathbb{S}^4$
endowed with the standard round metric $g_{rnd}$.

For any open subset $\Omega$ of $\mathbb{R}^{4}$ it is possible to construct standard OCS's, called 
\textit{constant}, in the following way: 
think $\mathbb{R}^{4}$ as the space of real quaternions $\mathbb{H}$ and define the set of imaginary units 
as follows
$$
\mathbb{S}:=\{x\in\mathbb{H}\,|\,x^2=-1\}.
$$
For any $q\in\mathbb{S}$, i
identifying each tangent space
 $T_{p}\Omega$ with $\mathbb{H}$ himself, we define the complex structure
 $\mathbb{J}_{q}$ everywhere by left 
 multiplication by $q$, i.e. $\mathbb{J}_{q}(p)v=qv$. 
 Any OCS defined globally on $\mathbb{H}$ is known to be constant (see~\cite[Proposition 6.6]{wood}), moreover
 it was proven in~\cite{salamonviac} the following result.

\begin{theorem}[\cite{salamonviac}, Theorem 1.3]
Let $J$ be an OCS of class $\mathcal{C}^1$ on $\mathbb{R}^{4}\setminus \Lambda$, where $\Lambda$
is a closed set of zero 1-dimensional Hausdorff measure. 
Then $J$ is the push-forward of the standard constant OCS on $\mathbb{R}^{4}$ under a conformal transformation
and $J$ can be maximally extended to the complement of a point $\mathbb{R}^{4}\setminus\{p\}$.
\end{theorem}

In the same paper it was proven the following result which completely solve the situation in a very particular case.

\begin{theorem}[\cite{salamonviac}, Theorem 1.6]\label{parabola}
 Let $\mathbb{J}$ be an OCS of class $\mathcal{C}^1$ on $\mathbb{R}^4\setminus \Lambda$, where $\Lambda$ is a round circle or a straight line, and assume
 that $\mathbb{J}$ is not conformally equivalent to a constant OCS. Then $\mathbb{J}$ is unique up to sign, and $\mathbb{R}^4\setminus \Lambda$ is a maximal domain
 of definition for $\mathbb{J}$.
\end{theorem}

In the hypotheses of Theorem~\ref{parabola} it is possible to construct explicitly the OCS $\mathbb{J}$ as follows. 
Under the identification $\mathbb{R}^4\setminus\Lambda\simeq\mathbb{H}\setminus\mathbb{R}$, 
a point $x$ can be written as 
$x=x_0+x_1i+x_2j+x_3k$, or as $x=\alpha+I_x\beta$, where 
$\alpha=x_0$ is the real part of $x$, 
$I_x=(x_1i+x_2j+x_3k)/\sqrt{x_1^2+x_2^2+x_3^2}\in\mathbb{S}$ and 
$\beta=\sqrt{x_1^2+x_2^2+x_3^2}$, so that $I_x\beta$ represents the imaginary part of $x$.
Then, for each $x=\alpha+I_x\beta\in\mathbb{H}\setminus\mathbb{R}$ we define $\mathbb{J}$ 
such as $\mathbb{J}(x)v=I_xv$, for each $v\in T_x(\mathbb{H}\setminus\mathbb{R})$.
This is an OCS over $\mathbb{H}\setminus\mathbb{R}$ that is constant on every
complex line
\begin{equation*}
 \mathbb{C}_I:=\{\alpha+I\beta\,|\,\alpha, \beta\in\mathbb{R}\},\quad I\in\mathbb{S},
\end{equation*}
but not globally constant, hence $\pm\mathbb{J}$  
are the only non-constant OCS's on this manifold (up to conformal transformations).

In~\cite{gensalsto} the authors proposed a new way to study the problem when $\Lambda\subset\mathbb{R}^4$ is a 
closed set of different type.
The idea is to take the OCS $\mathbb{J}$, previously defined, and to push it forward on the set 
we are interested in.
To do this we need to be sure that the function $f$, considered to push forward, preserves the
properties of $\mathbb{J}$.
%
%
This property holds true if $f$ is a quaternionic slice regular function defined on
a domain $\Omega\subset\mathbb{H}$ such that $\Omega\cap\mathbb{R}\neq \emptyset$.

\begin{definition}
 Let $\Omega\subset\mathbb{H}$ be a domain such that $\Omega\cap\mathbb{R}\neq\emptyset$
 and consider a function $f:\Omega\rightarrow\mathbb{H}$. For any $I\in\mathbb{S}$ we use the following notation:
 $\Omega_I:=\Omega\cap\mathbb{C}_I$ and $f_I:=f|_{\Omega_I}$. The function $f$
 is called \textit{slice regular} if, for each $I\in\mathbb{S}$, the following equation 
 holds,
 \begin{equation*}
  \frac{1}{2}\left(\frac{\partial}{\partial \alpha}+I\frac{\partial}{\partial \beta}\right)f_I(\alpha+I\beta)=0.
 \end{equation*}

\end{definition}

Examples of slice regular functions are polynomials and power series of the form
\begin{equation*}
 \sum_{k=0}^{+\infty} q^ka_k,\quad \{a_k\}_{k\in\mathbb{N}}\subset\mathbb{H},
\end{equation*}
defined in their convergence set.

If $\Omega$ is a domain of $\mathbb{H}$ such that $\Omega\cap\mathbb{R}=\emptyset$, then
the previous definition, by itself, is not enough to obtain a satisfactory theory of
regular functions. If, for instance $\Omega=\mathbb{H}\setminus\mathbb{R}$, then
it is possible to construct the following example: consider the function
$f:\mathbb{H}\setminus\mathbb{R}\rightarrow\mathbb{H}$ defined as
 \begin{equation*}
  f(x)=\begin{cases}
        1, &\mbox{if } x\in\mathbb{H}\setminus \mathbb{C}_{i}\\
        0, &\mbox{if } x\in \mathbb{C}_{i}\setminus \mathbb{R}.
       \end{cases}
 \end{equation*}
This function is of course regular but it is not even continuous. So regularity, by itself, does not 
imply even continuity. However this example is quite artificial since we could restrict
 to functions which are already differentiable. In section~\ref{slicesec}, we will
 show a more interesting example of differentiable regular function defined on $\mathbb{H}\setminus\mathbb{R}$
 that has similar problems.


%
To overcome this issue one can choose, 
\begin{enumerate}
 \item to study regular functions defined only over 
domains that do intersect the real axis;
\item to add some hypothesis to the set of 
functions.
\end{enumerate}
%
%
Since we are interested, among the other things, in extending the theory
to regular functions defined on a more general kind of domains,
 then we will use the second approach.

More precisely we will use the concepts of \textit{slice function} and of \textit{stem
function} introduced in~\cite{ghiloniperotti} in a more general context. Using these 
instruments (that will be defined in Section~\ref{slicesec}), it is possible to extend
some rigidity and differential results, that hold for regular functions defined on domains
which intersects the real axis (see~\cite{altavilla,altavilladiff}).

After a brief summary of the twistor theory of the 4-sphere and a review of the theory of slice regular functions, 
we will extend the theoretical work of 
\cite{gensalsto} in our setting of slice regular functions on domains without real points.
In particular, our point of view will be to describe the \textit{twistor interpretation} of the theory
of regular functions. With ``twistor interpretation'' we mean the correspondence given by the fact  that any slice regular function 
$f:\Omega\rightarrow\mathbb{H}$ lifts to a (holomorphic) curve
$\hat{f}:\mathcal{O}\subset\mathbb{CP}^3\rightarrow\mathbb{CP}^3$,
in the  space $\mathbb{CP}^3$ (see Section 4).
The complex projective space $\mathbb{CP}^3$ is in fact the \textit{twistor space} of $(\mathbb{S}^4\simeq\mathbb{H}\cup\{\infty\},g_{rnd})$, that is the total space
of a bundle parameterizing orthogonal almost complex structures on $\mathbb{S}^4$ and we let $\pi:\mathbb{CP}^3\rightarrow\mathbb{S}^4$ denotes
the twistor projection with fibre $\mathbb{CP}^1$.
It is a well known fact (see, for instance, \cite{salamonviac}), that a complex hypersurface in $\mathbb{CP}^3$ produces OCS's on subdomains of $\mathbb{S}^4$
wherever such a hypersurface is a single-valued graph with respect to the twistor projection, and that any OCS $J$ on a domain $\Omega$ generates a holomorphic hypersurface in $\mathbb{CP}^3$.

With this in mind, instead of giving examples of OCS's defined on some particular
domain, we will give classes of surfaces in the twistor space of $\mathbb{S}^{4}$
that can be described by means of slice regular functions (i.e.: that can be interpreted as the image of the lift of a slice regular function). One of the main results is that not all
 surfaces fit in this construction: first of all, they have to be ruled by lines.
Thanks to this peculiarity we found interesting to explore a little bit more the so-called
\textit{twistor transform}, which, given a slice regular function, returns 
the arrangement of lines whose lift carries on.
This construction is formalized in view of the work~\cite{shapiro}, where many properties
of submanifolds in the twistor space are studied
from the arrangement of lines lying on them.

At the end we will study a very particular case that fits very well in our theory.
We now briefly describe the structure of the present paper. In section~\ref{twsec} we summarize the main results in the twistor theory
of $\mathbb{S}^4$. In section~\ref{slicesec} we present
in a concise way the theory of slice regular functions. The only original part of this section is the one regarding 
\textit{slice affine functions}. Then, in section~\ref{twlift}, we effectively extend, to our more general context, the theory of~\cite{gensalsto},
prove the main theorems and analyse classes of surfaces up to degree 3 in $\mathbb{CP}^3$ that can
be reached by the lift of a slice regular function.
In section~\ref{ratcurvesec} we show that (almost) any rational curve in the Grassmannian
$\mathbb{G}r_2(\mathbb{C}^4)$ of 2-planes in $\mathbb{C}^4$ (interpreted as 
Pl\"ucker quadric in $\mathbb{CP}^5$), can be seen as the twistor transform
of a slice regular function. Then, we exhibit the set of slice regular functions whose
twistor transform describes a rational line inside $\mathbb{G}r_2(\mathbb{C}^4)$.
This result shows, in particular, the role of slice regular functions not defined on $\mathbb{R}$.
Indeed in the last remark of the section it is pointed out that, this set,
does not contain any slice regular function defined over the reals.

The subsequent section contains an explicit example of application.
This example provides some techniques that will be probably exploited in the future, in 
the study of much more  significant and technically complicated examples.
In some sense these explicit computations are natural because they regard 
the study of a particular degree-one surface (that is a hyperplane), in $\mathbb{CP}^3$,
hence a \textit{basic case} of study.

At the end there is a final small section on the possible future developments of the present work.


%


%


\section{Twistor space of $\mathbb{S}^4$}\label{twsec}
In this section we will review some aspect of twistor geometry focusing on the special 
case of $\mathbb{S}^{4}$. This part of the paper does not contain any new result but is
intended to be a summary of the main concepts and constructions that justify our
study. The themes that we are going to describe are classical but, according to our notation
and language, we refer to the following more recent papers~\cite{armpovsal,armsal,salamonviac,shapiro}.

The twistor space $Z$ of an oriented Riemannian manifold $(M,g)$ is the total space of a  bundle containing almost complex structures
(ACS) defined on $M$ and compatible with the metric $g$ and the orientation. The definition of the twistor
space does not depend on the full metric $g$ but only on its conformal class $[g]$. In fact, if 
$J$ is an ACS on $M$ compatible with $g$ and $g'=e^{f}g\in[g]$, then $J$ is obviously compatible
with respect to $g'$ as well.

A motivation to study twistor spaces is that, if $M$ is half-conformally-flat\footnote{Recall that a Riemannian metric
$g$ on $M$ is called \textit{half-conformally-flat}, or \textit{anti-self-dual}, if the self-dual part 
$W_{+}$ of the Weyl tensor vanishes. Vanishing of both self-dual and anti-self-dual part of the Weyl
tensor (i.e.: vanishing of the entire Weyl tensor), is equivalent to local conformal flatness of the metric $g$.
}, then its 
conformal geometry is encoded into the complex geometry of $Z$: for instance an ACS on $M$ compatible 
with the metric is integrable if and only if the corresponding section
of $Z$ defines a holomorphic submanifold.

Of course we will focus in the case in which $M$ is the 4-sphere $\mathbb{S}^{4}\simeq\mathbb{HP}^{1}$,
which topologically is $\mathbb{R}^{4}\cup\{\infty\}$. 
Here the quaternionic projective line $\mathbb{HP}^{1}$ is defined to be the set of equivalence classes
$[q_{1},q_{2}]$, where $[q_{1},q_{2}]=[pq_{1},pq_{2}]$ for any $p\in\mathbb{H}\setminus\{0\}$.
As we will see, the choice of left multiplication is forced by the choice of studying left slice functions.
Moreover, we embed the quaternionic space $\mathbb{H}$ into $\mathbb{HP}^{1}$ as $q\mapsto [1,q]$.
So, the point at infinity is represented by $[0,1]$. 
In this case, the twistor space is $\mathbb{CP}^{3}$ 
and the associated bundle structure $\pi:\mathbb{CP}^{3}\rightarrow \mathbb{HP}^{1}$ is the fibration,
defined as:
\begin{equation*}
\pi[X_{0},X_{1},X_{2},X_{3}]=[X_{0}+X_{1}j,X_{2}+X_{3}j].
\end{equation*}
It is known (see, e.g., \cite[Section 2.6]{salamonviac}), that any complex hypersurface  in $\mathbb{CP}^{3}$
transverse to the fibres of $\pi$
produces an OCS on subdomains of $\mathbb{R}^{4}$ whenever such a hypersurface is a single valued 
graph (with respect to the twistor projection). Vice versa, any OCS 
on a domain $\Omega\subset\mathbb{S}^{4}$ corresponds to a holomorphic hypersurface.
Moreover, for topological reason it is not possible to define any
ACS on the whole $\mathbb{S}^{4}$ (see~\cite[Proposition 6.6]{wood}), so, no hypersurface in 
$\mathbb{CP}^{3}$ can intersect every fibre of the twistor fibration in exactly one point.
\begin{remark}\label{matrix}
With our identifications, under the projection $\pi$, the matrix $J$ of the ACS corresponding to the point 
$[1,u=x+iy, X_{2},X_{3}]\in\mathbb{CP}^{3}$ is given by 
(up to notation and chirality, see~\cite[Section 2]{salamonviac}):
\begin{equation*}
J=\frac{-1}{1+|u|^{2}}\left(\begin{array}{cccc}
0 & 1-|u|^{2} & 2y  & -2x \\
-1+|u|^{2}&0  &-2x  & -2y \\
-2y&  2x&0  &1-|u|^{2}  \\
2x&  2y&  1-|u|^{2}&  0
\end{array}
\right).
\end{equation*}
\end{remark}

From the previous simple considerations it becomes natural to investigate the algebraic geometry of surfaces in 
$\mathbb{CP}^{3}$ from this perspective. For instance, a natural question that arises is to classify
surfaces of degree $d$ in complex projective space up to conformal transformations of the
base space $\mathbb{S}^{4}$. A starting point, in this framework, is to find \textit{conformal invariants},
after having clarified what we mean by \textit{conformal transformation} in the twistor
space $\mathbb{CP}^3$ of $\mathbb{S}^4$.

On any twistor fibre one can define a map $j$ which sends an ACS $J$ to $-J$. In our case $j$ is exactly
the action of multiplying a 1-dimensional complex subspace of $\mathbb{C}^{4}$ by the quaternion $j$
in order to get a new 1-dimensional space, i.e.: $j$ is the map on $\mathbb{CP}^{3}$ induced by the
quaternionic multiplication by $j$ in $\mathbb{HP}^{1}$:
\begin{equation*}
j:[X_{0},X_{1},X_{2},X_{3}]\mapsto[-\bar X_{1},\bar X_{0},-\bar X_{3},\bar X_{2}].
\end{equation*}
The map $j$ is an antiholomorphic involution of the twistor space with no fixed points. 
Starting with such a map $j$, one can recover the twistor fibration: given a point $X\in\mathbb{CP}^{3}$
there is a unique projective line connecting $X$ and $j(X)$. When $X$ varies, all these lines form the fibres. So, if a line $l$ in $\mathbb{CP}^{3}$ is a fibre for $\pi$, then $l=j(l)$.

The conformal symmetries of $\mathbb{S}^{4}$ correspond to the group of invertible transformations
\begin{equation*}
[q_{1},q_{2}]\mapsto[q_{1}d+q_{2}c,q_{1}b+q_{2}a],\quad a,b,c,d\in\mathbb{H},
\end{equation*}
where the invertibility condition is given by the following equation (see, e.g., \cite[Section 9.2]{genstostru}),
\begin{equation*}
|a|^{2}|d|^{2}+|b|^{2}|c|^{2}-2Re(b^{c} dc^{c}a)\neq 0.
\end{equation*}
Restricting to the affine line $q\in\mathbb{H}\mapsto[1,q]$, the latter becomes the linear fractional
transformation given by $q\mapsto (qc+d)^{-1}(qa+b)$.
These transformations correspond to the projective transformations of $\mathbb{CP}^{3}$ that preserve
$j$. We, therefore, say that two complex submanifolds of $\mathbb{CP}^{3}$ are \textit{conformally
equivalent} if they are projectively equivalent by a transformation that preserves $j$.

\begin{definition}
Let $\Sigma$ be an algebraic hypersurface of degree $d$ in $\mathbb{CP}^{3}$.
A \textit{twistor fibre} (or \textit{twistor line}) of $\Sigma$ is a fibre of $\pi$ which lies entirely 
within the surface $\Sigma$.

Moreover, we define the \textit{discriminant locus} of $\Sigma$ to be the set $D$ of points 
$p\in D\subset\mathbb{S}^{4}$, such that $\pi^{-1}(p)\cap\Sigma$ has cardinality different from $d$.
\end{definition}

The fibres of the twistor fibration are complex projective lines in $\mathbb{CP}^{3}$: if, in fact, we fix a 
quaternion $q=q_{1}+q_{2}j$, then the fibre $\mathbb{CP}^{1}=\pi^{-1}([1,q])$, is given by
\begin{equation*}
[1,q]=\pi[X_{0},X_{1},X_{2},X_{3}]=[X_{0}+X_{1}j,X_{2}+X_{3}j]=[1,(X_{0}+X_{1}j)^{-1}(X_{2}+X_{3}j)],
\end{equation*}
which translates into,
\begin{equation*}
(X_{0}+X_{1}j)(q_{1}+q_{2}j)=X_{2}+X_{3}j\,\Leftrightarrow\,
\begin{cases}
X_{2}=X_{0}q_{1}-X_{1}\bar q_{2}\\
X_{3}=X_{0}q_{2}+X_{1}\bar q_{1}
\end{cases}
\end{equation*}
The number
of twistor fibres of an algebraic surface $\Sigma$ is an invariant under conformal transformations.
Of course, if the degree of $\Sigma$ is $d$, then a generic fibre, intersecting $\Sigma$ transversely, 
will contain $d$ points because the defining polynomial of the surface, when restricted to the fibre,
gives a polynomial of degree $d$.

\begin{example}
The inversion $q\mapsto q^{-1}$ lifts to the automorphism of $\mathbb{CP}^{3}$ defined by:
\begin{equation*}
[X_{0},X_{1},X_{2},X_{3}]\mapsto[X_{2},X_{3},X_{0},X_{1}].
\end{equation*}
\end{example}

We will now review some classes of algebraic manifolds in $\mathbb{CP}^{3}$ already studied from this
point of view.

\subsection{Lines}
Consider two lines in $\mathbb{CP}^{3}$. If both lines are fibres of $\pi$ then they are conformal
equivalent by an isometry of $\mathbb{S}^{4}$ sending the image of one line under $\pi$ to the image
of the other line. If a line is not a fibre of $\pi$ then its image will be a round  2-sphere in $\mathbb{S}^{4}$
(corresponding to a 2-sphere or a 2-plane in $\mathbb{R}^{4}$). Given such a 2-sphere in $\mathbb{S}^{4}$,
there are two projective lines lying above it in $\mathbb{CP}^{3}$: in this case if $l$ is a line
projecting on the 2-sphere, then the other line is $j(l)$ which, in this case, is
disjoint from $l$ (see~\cite[Proposition 2.8]{shapiro}).
Therefore, a line in $\mathbb{CP}^{3}$ is given by either an oriented 2-sphere  or a point in $\mathbb{S}^4$.
Moreover, any two such 2-spheres are conformal equivalents (this geometric correspondence is explained
in detail in~\cite{shapiro}).

\subsection{Planes}
A plane in $\mathbb{CP}^{3}$ is given by a single linear equation of the form 
\begin{equation*}
c_{0}X_{0}+c_{1}X_{1}+c_{2}X_{2}+c_{3}X_{3}=0,
\end{equation*}
where, for each $i=1\dots4$, $c_{i}$ are constant numbers. A plane in $\mathbb{CP}^{3}$
cannot be transverse to every fibre of $\pi$ because it would then define a complex structure
on the whole $\mathbb{S}^{4}$ and (as already said), this is not possible.
Therefore, a plane, always contains at least one twistor fibre. Twistor fibres are always skew (otherwise they
would project to the same point), while
two lines in a plane always meet. Hence a plane always contains exactly one twistor fibre. If one picks
another line in the plane transverse to the fibre, its image under $\pi$ will be a 2-sphere.
We can find a conformal transformation of $\mathbb{S}^{4}$ mapping any 2-sphere with a marked point
to any other 2-sphere with a marked point (see again~\cite{shapiro}). We deduce that any couple of planes in $\mathbb{CP}^{3}$
are conformally equivalents (and not only projectively).
\subsection{Quadrics}
Non-singular quadrics in $\mathbb{CP}^{3}$  can be classified under conformal transformations of the
4-sphere $\mathbb{S}^{4}$.

\begin{theorem}[\cite{salamonviac}]
 Any non-singular quadric hypersurface in $\mathbb{CP}^3$ is equivalent under the action of the conformal group
 of $\mathbb{S}^4$ to the zero set of
 \begin{equation}\label{quaddiag}
  e^{\lambda+i\nu}X_0^2+e^{-\lambda+i\nu}X_1^2+e^{\mu-i\nu}X_2^2+e^{-\mu-i\nu}X_3^2,
 \end{equation}
or the zero set of
\begin{equation}\label{quadnondiag}
 i(X_0^2+X_1^2)+k(X_1X_3-X_0X_2)+X_1X_2-X_0X_3,
\end{equation}
where in the first case a couple of parameters $(\lambda,\mu,\nu)$, $(\lambda',\mu',\nu')$
define two quadrics in the same equivalence class if and only if $(\lambda,\mu,\nu)$ and $(\lambda',\mu',\nu')$ belong 
to the same orbit under the group $\Gamma$ of transformation of $\mathbb{R}^3$ generated by the four maps
\begin{equation*}
 \begin{cases}
  (\lambda,\mu,\nu)\mapsto(\lambda,\mu,\nu+\frac{\pi}{2})\\
  (\lambda,\mu,\nu)\mapsto(-\lambda,\mu,\nu)\\
  (\lambda,\mu,\nu)\mapsto(\lambda,-\mu,\nu)\\
  (\lambda,\mu,\nu)\mapsto(\mu,\lambda,-\nu),
 \end{cases}
\end{equation*}
while $k\in[0,1)$ is a complete invariant in the second case.
\end{theorem}

With this result the authors of~\cite{salamonviac} were able to describe the geometry of non-singular
quadric surfaces under the twistor projection $\pi$. 
\begin{theorem}[\cite{salamonviac}]\label{classification}
For any non-degenerate quadric surface $\mathcal{Q}\subset\mathbb{CP}^{3}$ there are three possibilities.
\begin{enumerate}
 \item $\mathcal{Q}$ is a real quadric with discriminant locus a circle in $\mathbb{S}^4$ and $\mathcal{Q}$ contains all
 the twistor lines over the circle. 
  \item $\mathcal{Q}$ contains exactly one or exactly two twistor lines. In these cases the discriminant locus is a singular torus pinched at
 one or two points, respectively.
 \item $\mathcal{Q}$ does not contain any twistor lines. In this case the discriminant locus is a torus $\mathbb{T}^2\subset\mathbb{S}^4$
 with a smooth unknotted embedding.
 \end{enumerate}
Moreover if $\mathcal{Q}$ is the zero locus of the polynomial in~\eqref{quaddiag} with $0\leq\lambda\leq\mu$ and $0\leq\nu<\pi/2$, then
\begin{enumerate}
 \item $\mathcal{Q}$ contains a family of twistor lines over a circle if and only if $\lambda=\mu=\nu=0$,
 \item $\mathcal{Q}$ contains exactly two twistor lines if and only if $\lambda=\mu\neq 0$ and $\nu=\pi/2$,
 \item $\mathcal{Q}$ contains no twistor lines in the other cases.
\end{enumerate}
Finally if $\mathcal{Q}$ is the zero locus of the polynomial in~\eqref{quadnondiag} with $k\in[0,1)$, then the corresponding
quadric $\mathcal{Q}$ contains exactly one twistor line.
\end{theorem}

Singular quadric surfaces are still not studied.

\subsection{Cubics}
In~\cite{armpovsal} it is proven that a non-singular cubic contains at most 5 twistor lines. Moreover
for a generic set of 5 points lying on a 2-sphere in $\mathbb{S}^{4}$ there exists a one parameter family
of projectively isomorphic but conformal non-isomorphic non-singular cubic surfaces with 5 twistor lines 
corresponding to the 5 points.
The following result was proven in~\cite{armpovsal} where the authors begin the study of this topic for
non-singular cubic surfaces.
\begin{theorem}[\cite{armpovsal}]
Given 5 points on a 2-sphere in $\mathbb{S}^{4}$, there is a non-singular cubic surface with 5 twistor lines 
corresponding to these points if and only if no 4 of the points lie on a circle.
\end{theorem}

\subsection{Quartics}
The only known result for quartic surfaces in this direction (in our knowledge), regards a singular quartic scroll studied in~\cite[Section 7]{gensalsto}. This quartic scroll $\mathcal{K}$ is defined by the following
equation
\begin{equation*}
(X_{1}X_{2}-X_{0}X_{3})^{2}+2X_{1}X_{0}(X_{1}X_{2}+X_{0}X_{3})=0.
\end{equation*}
Define now $\gamma$ to be the parabola
\begin{equation*}
\gamma:=\{t^{2}+ti\,|\,t\in\mathbb{R}\}\subset\mathbb{C}_{i},
\end{equation*}
and $\Gamma$ the following paraboloid of revolution:
\begin{equation*}
\Gamma:=\{q_{0}+jq_{2}+kq_{3}\,|\,q_{0},q_{2},q_{3}\in\mathbb{R},\,q_{0}=\frac{1}{4}-(q_{2}^{2}+q_{3}^{2})\}.
\end{equation*}
The following theorem is the mentioned result.
\begin{theorem}[\cite{gensalsto}]
Let $q\in\mathbb{H}$. The cardinality of the fibre $\pi^{-1}(q)\cap\mathcal{K}$ is different
from 4 in the following cases:
\begin{enumerate}
\item $q\in\gamma$ if and only if $\pi^{-1}(q)\subset\mathcal{K}$;
\item $q\in\mathbb{C}_{i}\setminus\gamma$ if and only if $\pi^{-1}(q)$ contains exactly two
singular points of $\mathcal{K}$;
\item $q\in\Gamma\setminus \{\frac{1}{4}\}$ if and only if $\pi^{-1}(q)$is tangent to $\mathcal{K}$ at two
smooth points.
\end{enumerate}
\end{theorem}

This last case was ``reverse engineered'' from the study of the possible (non-constant) OCS's 
defined on $\mathbb{R}^{4}\setminus\lambda$. 
In all the previous cases, given a surface (or a family of surfaces), we have described its conformal geometry.
As already said in the introduction, we are interested in the relation between the class of slice regular functions 
and the geometry of submanifolds in $\mathbb{CP}^{3}$ and, therefore, later we will switch our attention
in this sense.
%
%
%
%

\section{Slice regular functions}\label{slicesec}
We have already defined in the introduction the sphere of imaginary units $\mathbb{S}\subset\mathbb{H}$
and the family of \textit{slices} $\mathbb{C}_{I}\subset\mathbb{H}$.
Given $x=\alpha+I\beta\in\mathbb{H}\setminus\mathbb{R}$ we also define the sphere 
$$\mathbb{S}_{x}:=\{y\in\mathbb{H}\,|\,y=\alpha+J\beta,\,J\in\mathbb{S}\},$$
and, for any $I\in\mathbb{S}$, the family of \textit{semislices} as
$$\mathbb{C}_{I}^{+}:=\{x\in\mathbb{H}\,|\,x=\alpha+I\beta,\,\alpha\in\mathbb{R}, \beta\geq 0\}.$$
%
%
If $x=\alpha+I\beta$ is a quaternion, its \textit{usual conjugation} will be denoted by $x^{c}=\alpha-I\beta$.

Given a domain $\Omega\subset \mathbb{H}$, a  function $f:\Omega\rightarrow\mathbb{H}$
is said to be \textit{Cullen regular} if, for any $I\in\mathbb{S}$, the restriction 
$f|_{\Omega\cap\mathbb{C}_{I}}$ is a holomorphic function with respect to $I$.
In other words if, for any $I\in\mathbb{S}$ the following equation holds
\begin{equation}\label{eqCul}
\frac{1}{2}\left(\frac{\partial}{\partial\alpha}+I\frac{\partial}{\partial\beta}\right)f|_{\Omega\cap\mathbb{C}_{I}}\equiv 0.
\end{equation}

The theory of Cullen regular functions was born to include polynomials and power series
of the form $\sum_{k\in\mathbb{N}}q^{k}a_{k}$, where $\{a_{k}\}_{k\in\mathbb{N}}\subset\mathbb{H}$. 
This theory, introduced by G. Gentili 
and D. Struppa in~\cite{gentilistruppa} and based on a definition by C. Cullen (see~\cite{cullen}),
is revealing, in the last years, to be very rich and interesting both from a theoretical point
of view and (as this paper and~\cite{gensalsto} show), from the point of view of the applications.

Even though at first there was an explosion of results regarding, for instance, the rigid behaviour of 
such regular functions 
\cite{colgensabstru,bohrsarfatti,blochlandau,landautoeplitz,gentilistoppato,gentilistoppato2}, 
and the possibility of expanding them in different kind of power series
\cite{genstoseries,stoppato,stoppatopoles}, the formalism used to introduce the theory turns out to be inadequate to study such functions, when defined over particular
domains: namely, domains with empty intersection with the real axis.
The most simple example of what could go wrong is the function 
$f:\mathbb{H}\setminus\mathbb{R}\rightarrow \mathbb{H}$ defined to be equal to
some constant $q_{0}\in\mathbb{H}$ everywhere but on a fixed slice $\mathbb{C}_{J}$ on which
it is set to be equal to some different constant $q_{1}\in\mathbb{H}$ (compare with the function defined in the introduction). 
Such functions of course satisfy Equation~\eqref{eqCul}, if restricted to any slice,
but are not even continuous.
Later in this section, when the key features of slice regularity will be outlined, we will 
show another most significative example of a class $\mathcal{C}^{\infty}$ function which satisfies the 
definition of Cullen regularity but, for some reason, we do not want in our theory (see Example~\ref{ellips}).
Therefore, to describe the theory of slice regular functions on general domains, we will adopt 
a different approach.
%
The approach that we will use is the one introduced by
R. Ghiloni and A. Perotti in~\cite{ghiloniperotti}, which exploit the use of the so-called 
\textit{stem functions} to define the class of continuous functions to which we will
apply the definition of regularity.
The use of stem functions might seems unnecessarily technical, however, by using
precisely these techniques many results in the theory of slice regular functions were 
extended to a more general setting and many others were proven for the first time
\cite{altavillaPhD,altavilla,altavilladiff,ghiloniperotti,ghiloniperotti2,ghiloniperotti3}.
We will, then, describe our family of functions, directly by using this approach.

Let $\mathbb{H}_{\mathbb{C}}$ denote the real tensor product $\mathbb{H}\otimes_\mathbb{R} \mathbb{C}$. An element of $\mathbb{H}_{\mathbb{C}}$ will be of the form 
$p=x+\sqrt{-1}y$, where $x$ and $y$ are quaternions. Given another element $q=z+\sqrt{-1}t$ in
$\mathbb{H}_{\mathbb{C}}$, we define the following product,
\begin{equation*}
pq=xz-yt+\sqrt{-1}(xt+yz).
\end{equation*}
Of course $\sqrt{-1}$ plays the role of a complex structure in $\mathbb{H}_{\mathbb{C}}$.
With the previous product, the space $\mathbb{H}_{\mathbb{C}}$ results to be a complex alternative
algebra with unity.
Given an element $p=x+\sqrt{-1}y\in\mathbb{H}_{\mathbb{C}}$ we define the following two commuting
conjugations: 
\begin{itemize}
\item $p^{c}=x^{c}+\sqrt{-1}y^{c}$;
\item $\bar p=\bar x+\sqrt{-1}\bar y$.
\end{itemize}

\begin{definition}[Stem function]
Given a domain $D$ in $\mathbb{C}$ a function $F:D\rightarrow \mathbb{H}_{\mathbb{C}}$
is said to be a \textit{stem function} if, for any $z\in D$ such that $\bar z\in D$
one has that $F(\bar z)=\overline{F(z)}$.
\end{definition}
The condition appearing in the previous definition translates in the following way: a function $F$ from a domain 
$D\subset\mathbb{C}$ to $\mathbb{H}_{\mathbb{C}}$ can be represented as 
$F(\alpha+i\beta)=F_{1}(\alpha+i\beta)+\sqrt{-1}F_{2}(\alpha+i\beta)$, with $F_{1}$ and $F_{2}$
quaternionic-valued functions;
 then $F$ is a stem function if $F_{1}$ and $F_{2}$ are even and odd
with respect to $\beta$, respectively.
For this reason there is no loss of generality in taking $D$ to be symmetric with respect to the real axis.

We will say that a stem function $F=F_{1}+\sqrt{-1}F_{2}$ has a certain regularity (e.g.: is of
class $\mathcal{C}^{n}$, for some $n\in\mathbb{N}\cup\{\infty,\omega\}$), if its components
have that regularity.
%

\begin{definition}[Circularization] Given a set $D\subset\mathbb{C}$ we define its \textit{circularization}
as the subset  $\Omega_{D}$ of $\mathbb{H}$ determined by the following equality
\begin{equation*}
\Omega_{D}:=\{\alpha+I\beta\in\mathbb{H}\,|\,\alpha+i\beta\in D,\,I\in\mathbb{S}\}.
\end{equation*}
If $\Omega\subset\mathbb{H}$ is such that 
$\Omega=\cup_{x\in\Omega}\mathbb{S}_{x}$ and $\Omega=\Omega_{\Omega\cap \mathbb{C}_{I}}$ for any $I\in\mathbb{S}$, then it will be called a \textit{circular set}.
For any $I\in\mathbb{S}$ we will use the following notation: $D_I=\Omega_{D}\cap\mathbb{C}_I$. 
\end{definition}

\begin{remark}
If $D$ is symmetric with respect to the real axis and $D\cap\mathbb{R}=\emptyset$, then 
$\Omega_{D}\simeq D^{+}\times \mathbb{S}$, where $D^{+}$ is the intersection 
between $D$ and the complex upper half plane: $D^{+}=D\cap \mathbb{C}^{+}$.
\end{remark}

\begin{definition}[Slice function]
Let $\Omega_{D}$ be a circular set in $\mathbb{H}$. A function $f:\Omega_{D}\rightarrow\mathbb{H}$
is said to be a \textit{(left) slice function} if it is induced by a stem function $F=F_{1}+\sqrt{-1}F_{2}$ (denoted 
by $f=\mathcal{I}(F)$), in the following way:
\begin{equation*}
f(\alpha+I\beta)=F_{1}(\alpha+i\beta)+IF_{2}(\alpha+i\beta),\quad\forall\alpha+I\beta\in\Omega_{D}.
\end{equation*}
The family of slice functions defined over some circular set $\Omega_{D}$ will be denoted by 
$\mathcal{S}(\Omega_{D})$.
We will say that a slice function $f=\mathcal{I}(F)$ has a certain regularity (e.g.: is of class $\mathcal{C}^{n}$, for some 
$n\in\mathbb{N}\cup\{\infty,\omega\}$) if the inducing stem function $F$ has that regularity.
The space of slice functions of class $\mathcal{C}^{n}$ defined over $\Omega_{D}$ will be denoted 
by $\mathcal{S}^{n}(\Omega_{D})$.
\end{definition}
For each $n$, the family $\mathcal{S}^n(\Omega_{D})$ is a real vector space and a quaternionic right module, i.e.:
for any $f,g\in\mathcal{S}^n(\Omega_{D})$, for any $c\in\mathbb{R}$ and for any $q\in\mathbb{H}$,
$cf+gq\in\mathcal{S}^n(\Omega_{D})$.

Of course one can define analogously \textit{right} slice functions, by putting, in the previous
definition, the complex imaginary unit at the right of $F_{2}$. The uprising theory would be 
completely symmetric with respect to the one described here.

The nature of stem functions yields the well-posedness of the slice function's definition. In fact if 
$f=\mathcal{I}(F_{1}+\sqrt{-1}F_{2}):\Omega_{D}\rightarrow\mathbb{H}$ is a slice function and $x=\alpha+I\beta$ is a quaternion
in $\Omega_{D}$, then $f(\alpha+(-I)(-\beta))=F_{1}(\alpha-i\beta)-IF_{2}(\alpha-i\beta)=F_{1}(\alpha+i\beta)+IF_{2}(\alpha+i\beta)=f(\alpha+I\beta)$. 
Moreover one can see, from the definition, that a (left) slice function is nothing but a quaternionic 
function of one quaternionic variable that is \textit{quaternionic (left) affine with respect to the imaginary unit}.
This fact, together with the next (basic but) fundamental result, shows that there is no loss
of generality in choosing circular sets as domains of definition for slice functions and that,
for any slice function there is a unique inducing stem function.

\begin{theorem}[Representation formula, \cite{ghiloniperotti}, Proposition 6]
Let $f\in\mathcal{S}(\Omega_{D})$ be a slice function defined over any circular set $\Omega_{D}$.
Then, for any $J\neq K\in\mathbb{S}$, $f$ is uniquely determined by its values over $\mathbb{C}_{J}^{+}$
and $\mathbb{C}_{K}^{+}$ by the following formula:
\begin{equation*}
f(\alpha+I\beta)=(I-K)(J-K)^{-1}f(\alpha+J\beta)-(I-J)(J-K)^{-1}f(\alpha+K\beta),\,\forall\,\alpha+I\beta\in\Omega_{D}.
\end{equation*}
 In particular if $K=-J$, we get the following simpler formula
\begin{equation}\label{repreform}
f(x)=\frac{1}{2}\left[f(\alpha +J\beta )+f(\alpha -J\beta )-IJ\left(f(\alpha +J\beta )-f(\alpha -J\beta )\right)\right].
\end{equation}
\end{theorem}

This well-known result can be easily proven having in mind that a straight line parametrized by an affine 
function in an affine space can be recovered simply by two of its values.
%

Given any slice function, it is possible to define its spherical derivative as follows.
\begin{definition}[Spherical derivative]
Let $f=\mathcal{I}(F_{1}+\sqrt{-1}F_{2})\in\mathcal{S}(\Omega_{D})$ and $x=\alpha+I\beta\in\Omega_{D}\setminus\mathbb{R}$.
The \textit{spherical derivative} of $f$ is the slice function $\partial_{s}f$ induced by the stem function
$F_{2}(z)/Im(z)$.
\end{definition}

For any point $x\in\Omega_{D}\setminus\mathbb{R}$, the spherical derivative of a slice function $f$ can be also defined as
\begin{equation*}
\partial_{s}f(x)=\frac{1}{2}Im(x)^{-1}(f(x)-f(x^{c})).
\end{equation*}
The spherical derivative of any slice function is constant on each sphere $\mathbb{S}_{x}$ contained
in the domain of definition of the function. Moreover, if the function $f$ is of class at least $\mathcal{C}^{1}$,
then its spherical derivative can be extended continuously to the real line (see~\cite[Proposition 7]{ghiloniperotti}).

As the reader can see, the spherical derivative of a slice function is not a genuine derivative, i.e. it is
not defined as some sort of limit of incremental ratio. However, as we will see in Theorem~\ref{reprediff}, it is useful to control the
behaviour of a slice function alongside the spheres $\mathbb{S}_{x}$ contained in its domain of definition.
In this view, we are going to define now the partial derivatives along the remaining directions,
which are the slices. First observe that, for a sufficiently regular stem function $F=F_{1}+\sqrt{-1}F_{2}$,
if $z=\alpha+i\beta$ is a point in the domain of $F$, 
the partial derivatives $\partial F_{1}/\partial \alpha$ and $\sqrt{-1}(\partial F_{2}/\partial\beta)$
are stem functions too.

\begin{definition}[Slice Derivative]
Given a function $f\in\mathcal{S}^{1}(\Omega_{D})$ we define its \textit{slice derivatives} as the following
continue slice functions defined over $\Omega_{D}$:
\begin{equation*}
\begin{array}{c}
\displaystyle\frac{\partial f}{\partial x}:=\mathcal{I}\left(\displaystyle\frac{\partial F}{\partial z}=\displaystyle\frac{1}{2}\left(\displaystyle\frac{\partial F}{\partial \alpha}-\sqrt{-1}\displaystyle\frac{\partial F}{\partial \beta}\right)\right)\\
\displaystyle\frac{\partial f}{\partial x^{c}}:=\mathcal{I}\left(\displaystyle\frac{\partial F}{\partial \bar z}=\displaystyle\frac{1}{2}\left(\displaystyle\frac{\partial F}{\partial \alpha}+\sqrt{-1}\displaystyle\frac{\partial F}{\partial \beta}\right)\right)
\end{array}
\end{equation*}
\end{definition}

As we said before, $\sqrt{-1}$ is a complex structure for $\mathbb{H}_{\mathbb{C}}$. But then a stem function $F:D\rightarrow\mathbb{H}_{\mathbb{C}}$ is holomorphic
if $\partial F/\partial \bar z=0$. In this way we naturally define slice regularity as follows.
\begin{definition}[Slice regularity]
A (left) slice function $f=\mathcal{I}(F)\in\mathcal{S}^{1}(\Omega_{D})$ is said to be \textit{(left) slice regular}
if the slice derivative $\partial f/\partial x^{c}$ vanishes everywhere. The set of slice regular functions
defined over a certain domain $\Omega_{D}$ will be denoted by $\mathcal{SR}(\Omega_{D})$.
\end{definition}
The set $\mathcal{SR}(\Omega_{D})$ is a real vector space and quaternionic right module. 
Again, the theory does not change if we consider right slice regular function instead of left ones.
\begin{remark}\label{charareg}
Due to the Representation Formula, a slice function $f$ is slice regular if and only if
for any fixed $I\in\mathbb{S}$, Equation~\eqref{eqCul} holds.
Moreover, if a function $f$ is slice regular, then its slice derivative $\partial f/\partial x$ is regular as well.
All these properties are discussed in~\cite{ghiloniperotti}.
\end{remark}

There is a formula that links the value of the spherical and the slice derivatives.
\begin{proposition}[\cite{altavilladiff}, Proposition 12]
Let $f\in\mathcal{SR}(\Omega_{D})$ be a slice regular function, then the following formula holds:
\begin{equation*}
\frac{\partial f}{\partial x}(x)=2Im(x)\left(\frac{\partial}{\partial x}\partial_{s}f\right)(x)+\partial_{s}f(x),\quad\forall x=\alpha+J\beta\in\Omega_{D}.
\end{equation*}
\end{proposition}

Any slice regular function is Cullen regular (see~\cite[Definition 1.1]{genstostru}), but, if 
the domain of definition does not intersects the real line then the converse is not true in general.
This issue was studied in~\cite{ghiloniperotti2}, where the authors show that asking for
a generic quaternionic function defined over a domain without real points to be Cullen regular is not enough 
to obtain a satisfactory theory:
 in general you lose \textit{sliceness} that is equivalent to losing the
Representation Formula and so, many fundamental theorems 
in this theory have no chance to hold in this very general context.
However the author of the present paper believes that this issue should be further studied: 
some interesting subclasses might arise from this investigation.

%
%
%

\begin{example}\label{ellips}
 Fix a real number $\lambda\notin\{-1,0,1\}$. Let $x=\alpha+I_x\beta$ be any non-real quaternion and 
define $f:\mathbb{H}\setminus\mathbb{R}\rightarrow\mathbb{H}$ as
 \begin{equation*}
  f(x)=I_x+\lambda i I_xi.
 \end{equation*}
 The function $f$ is a class $\mathcal{C}^{\infty}$ quaternionic functions of one quaternionic variable that is Cullen regular but not 
slice regular. 
In fact, as it is pointed out in~\cite[proof of Proposition 6]{ghiloniperotti}, for any $J\neq K \in\mathbb{S}$
the odd part of the stem function inducing $f$ would be
$F_{2}(\alpha+i\beta)=(J-K)^{-1}(f(\alpha+J\beta)-f(\alpha+K\beta))$. But, computing $F_{2}$ for $J=j$
and $K=k$, one easily obtains $F_{2}(z)=1-\lambda$ that is not odd.

For sake of completeness, if $I_{x}=ai+bj+ck$, then, for any $x\in\mathbb{H}\setminus\mathbb{R}$, the image of $\mathbb{S}_{x}$ is the pure imaginary ellipsoid parametrized as
 $$
 a(1-\lambda)i+b(1+\lambda)j+c(1+\lambda)k,
 $$
 with $a^{2}+b^{2}+c^{2}=1$. 
%
%
%
%
\end{example}

A consolidated and well known result about slice regular functions is the \textit{Splitting Lemma}. It says that any slice regular function, if properly
restricted, admits a splitting into two
actual complex holomorphic functions. A proof of this result can be found in~\cite{colgensabstru, ghiloniperotti3}, the first with the additional hypothesis that 
the domain of definition intersects the real axis.

\begin{lemma}[\cite{colgensabstru,ghiloniperotti3}]\label{splitting}
 Let $f\in\mathcal{SR}(\Omega_D)$. Then, for each $J\in\mathbb{S}$ and each $K\bot J$, $K\in\mathbb{S}$, there exist two holomorphic functions
 $G,H:D_J\rightarrow \mathbb{C}_J$ such that 
 \begin{equation*}
  f_J=G+HK.
 \end{equation*}
\end{lemma}

Observe that $G$ and $H$ are defined over the whole $D_{J}$. This means that, if $D_{J}$
is disconnected and the disjoint union of $D_1$ and $D_2$, then, $G$ and $H$ could have unrelated different behaviour on $D_1$ and $D_2$. A particular case is when 
$\Omega_{D}\cap\mathbb{R}=\emptyset$, where, \textit{a priori} the function $G$ and $H$ can
have  different behaviours if restricted either to $D_{J}^{+}$ or $D_{J}^{-}$.

Speaking now about operations between slice functions,
in general, their pointwise product is not a slice 
function\footnote{For instance, if $f(q)=qa$ and $g(q)=q$, with $a\in\mathbb{H}\setminus\mathbb{R}$, then $h(q)=f(q)g(q)=qaq$ is not a slice function.}.
However, there exists another notion of product which works well in our context.
The following, introduced in~\cite{colgensabstru,gentilistoppato2} for slice regular functions defined over domains that
do intersect  $\mathbb{R}$ and in~\cite[Definition 9]{ghiloniperotti} for slice functions (in the context of real alternative algebras),
is the notion that we will use. 
\begin{definition}[Slice product]
Let $f=\mathcal{I}(F)$, $g=\mathcal{I}(G)$ $\in\mathcal{S}(\Omega_{D})$ the \textit{(slice) product} of $f$ and $g$ is the slice function
\begin{equation*}
f\cdot g:=\mathcal{I}(FG)\in\mathcal{S}(\Omega_{D}).
\end{equation*}
\end{definition}
 Explicitly, if $F=F_{1}+\sqrt{-1}F_{2}$ and $G=G_{1}+\sqrt{-1}G_{2}$ are stem functions, then $FG=F_{1}G_{1}-F_{2}G_{2}+\sqrt{-1}(F_{1}G_{2}+F_{2}G_{1})$.
 \begin{remark}
   Let $f(x)=\sum_jx^ja_j$ and $g(x)=\sum_kx^kb_k$ be polynomials or, more generally, converging power series with coefficients $a_j,b_k\in\mathbb{H}$. The usual product of polynomials can be extended to power series in the following way: the \textit{star product} $f* g$ of $f$ and $g$ is the convergent power series
 defined by setting
\begin{equation*}
 (f*g)(x):=\sum_nx^n\left(\sum_{j+k=n}a_jb_k\right).
\end{equation*}
In~\cite[Proposition 12]{ghiloniperotti} it was proven that the product of $f$ and $g$, viewed as slice functions, coincide with the star product $f*g$, i.e.: 
$\mathcal{I}(FG)=\mathcal{I}(F)*\mathcal{I}(G)$.
Indeed sometimes the slice product between $f$ and $g$ is denoted by $f*g$ (see~\cite{gentilistoppato} or~\cite{gentilistruppa}) and called \textit{regular product},
to stress the fact that this notion of product was born to preserve the regularity. The next proposition precises 
this fact.
 \end{remark}

 \begin{proposition}[\cite{ghiloniperotti}, Proposition 11]\label{starprod}
 If $f,g\in\mathcal{SR}(\Omega_{D})$ then $f\cdot g\in \mathcal{SR}(\Omega_{D})$.
 \end{proposition}
In~\cite{ghiloniperotti} it is also pointed out and proved  that the regular product introduced in 
\cite{colgensabstru,gentilistoppato2}
is generalized by this one if the domain $\Omega_D$ does not have real points.
An idea to prove Proposition~\ref{starprod} is simply to explicit the slice product in term of stem
functions and compute the Cauchy-Riemann equations.

 
 The slice product of two slice functions coincides with the pointwise product if the first slice function is \textit{real}
 (see~\cite[Definition 10]{ghiloniperotti}).
\begin{definition}[Real slice function]
A slice function $f=\mathcal{I}(F)$ is called \textit{real} or \textit{slice-preserving} or, even, \textit{quaternionic intrinsic} if the $\mathbb{H}$-valued components $F_{1}$, $F_{2}$ are real valued.
\end{definition}

The next proposition, stated in~\cite[Lemma 6.8]{ghilmorper}, justifies the different names given in the previous definition.
\begin{proposition}
Let $f=\mathcal{I}(F)\in\mathcal{S}(\Omega_{D})$ be a slice function. The following conditions are equivalent.
\begin{itemize}
\item $f$ is real.
\item For all $J\in \mathbb{S}$, $f(D_J)\subset \mathbb{C}_{J}$.
\item For all $x$ in the domain of $f$ it holds $f(x)=(f(x^{c}))^{c}$.
\end{itemize}
\end{proposition}
These functions are special since, in a certain sense, they carry the concept of complex
 function in our setting. In fact, if $h(z)=u(z)+iv(z)$ is a complex function defined
 over a certain domain $D\subset\mathbb{C}$ with $D\cap\mathbb{R}\neq\emptyset$, then the function $H:D\rightarrow \mathbb{H}_{\mathbb{C}}$ defined as $H(z)=u(z)+\sqrt{-1}v(z)$ is a stem function,
and $\mathcal{I}(H)$ is a real slice function.

As stated in~\cite{gentilistoppato2}, if
$f$ is a slice regular function defined on $B(0,R)$, the ball of centre zero and radius $R$ for
some $R>0$, then $f$ is real if and only if it can be expressed as a power series of the form
\begin{equation*}
f(x)=\sum_{n\in\mathbb{N}}x^{n}a_{n},
\end{equation*}
with $a_{n}$ real numbers.

By a simple computation, it is possible to prove the following lemma.

\begin{lemma}[\cite{altavilla}, Lemma 2.12]
Let $f=\mathcal{I}(F),g=\mathcal{I}(G)\in\mathcal{S}(\Omega_{D})$, with $f$ real, then the slice function $h:\Omega_{D}\rightarrow \mathbb{H}$, defined by
 $h:=f\cdot g$ is such that
\begin{equation*}
h(x)=f(x)g(x).
\end{equation*}
\end{lemma}
%

Now, we are going to define an ``inversion'' for slice functions.
The following first two definitions appeared for the first time in~\cite{colgensabstru}, 
 can be found also in~\cite{gentilistoppato} and~\cite{gentilistoppato2}. Later they were
generalized by Ghiloni and Perotti for slice functions in~\cite[Definition 11]{ghiloniperotti}.
The definition of slice reciprocal was firstly introduced in~\cite{colgensabstru, gentilistoppato2, gentilistoppato,stoppatopoles}
and then in~\cite{altavilla} if the domain of definition has empty intersection with the real line. 
Let us denote the zero set of a slice function $f$ by $V(f)$.
\begin{definition}[Slice conjugate, Normal function, Slice reciprocal]
Let $f=\mathcal{I}(F)\in\mathcal{S}(\Omega_{D})$, then also $F(z)^{c}:=F_{1}(z)^{c}+\sqrt{-1}F_{2}(z)^{c}$ is a stem function. We define the following three functions.
\begin{itemize}
\item $f^{c}:=\mathcal{I}(F^{c})\in\mathcal{S}(\Omega_{D})$, called \textit{slice conjugate} of $f$.
\item $N(f):=f^c\cdot f$ \textit{symmetrization} or \textit{normal function} of $f$ (the symmetrization of $f$ is sometimes denoted by $f^s$).
\item If $f=\mathcal{I}(F)$ is slice regular, we call the \textit{slice reciprocal} of $f$ the slice function
\begin{equation*}
f^{-\cdot}:\Omega_{D}\setminus V(N(f))\rightarrow \mathbb{H}
\end{equation*}
defined by	
\begin{equation*}
f^{-\cdot}=\mathcal{I}((F^cF)^{-1}F^c).
\end{equation*}
\end{itemize}
\end{definition}

From the previous definition it follows that, if $x\in\Omega_D\setminus V(N(f))$, then

\begin{equation*} 
f^{-\cdot}(x)=(N(f)(x))^{-1}f^{c}(x).
\end{equation*}

Various results about the previous functions were reviewed in~\cite{altavillaPhD}, including
the fact that if a function $f$ is slice regular then $f^{c}$, $N(f)$ and $f^{-\cdot}$ (where
it is defined) are all regular.
The notion of slice reciprocal was engineered so that, if $f$ is a slice regular function
with empty zero-locus then,
\begin{equation*}
f\cdot f^{-\cdot}=f^{-\cdot}\cdot f=1.
\end{equation*}
For more information about this result (including its proof) see, again,~\cite{altavillaPhD}.
We are going now to review the construction of slice forms
introduced by the author of the present paper in~\cite{altavillaPhD, altavilla}.
We will start with the following general definition.
\begin{definition}\label{slcfrm}
 Let $f=\mathcal{I}(F)\in\mathcal{S}^1(\Omega_D)$. We define the \textit{slice differential} $d_{sl}f$ of $f$ as the following differential form:
 \begin{equation*}
 \begin{array}{rrcc}
     d_{sl}f: &  (\Omega_D\setminus \mathbb{R})  & \rightarrow & \mathbb{H}^*,\\
   & \alpha+I\beta &  \mapsto  & dF_1(\alpha+i\beta)+IdF_2(\alpha+i\beta).
 \end{array}
    \end{equation*}
\end{definition}
\begin{remark}
The one-form $\omega:\mathbb{H}\setminus\mathbb{R}\rightarrow\mathbb{H}^*$ defined as $\omega(\alpha+I\beta)=Id\beta$, 
represents the outer radial direction to the sphere $\mathbb{S}_x=\{\alpha+K\beta\,|\,K\in\mathbb{S}\}$.
Then $\omega(\alpha+I(-\beta))=\omega(\alpha+(-I)\beta)=-\omega(\alpha+I\beta)$.
We can translate this observation in the language of slice forms.
The function $h(x)=Im(x)$ is a slice function induced by $H(z)=\sqrt{-1}Im(z)$.
Then we have $d_{sl}h(\alpha+I\beta)=Id\beta(\alpha+i\beta)$ and, thanks to the previous considerations
$d_{sl}h(\alpha+(-I)(-\beta))=-Id\beta(\alpha-i\beta)=Id\beta(\alpha+i\beta)$.
Summarizing, we have that $d\beta(\bar z)=-d\beta(z)$.
The same does not hold for $d\alpha$ which is a constant one-form over $\mathbb{H}$ and 
for this reason in the next computations
we will omit the variable (i.e.: $d\alpha=d\alpha(z)=d\alpha(\bar z)$).
\end{remark}
In~\cite[Proposition 10]{altavilladiff} it is proved that
Definition~\ref{slcfrm} is well posed, i.e. if $D$ is symmetric with respect to the real axis, then
 \begin{equation*}
  d_{sl}f(\alpha+I\beta)=d_{sl}f(\alpha+(-I)(-\beta)),\quad \forall \alpha+I\beta\in\Omega_D\setminus\mathbb{R}
 \end{equation*}
We can represent, then, the slice differential as follows.
\begin{proposition}[\cite{altavilladiff}, Proposition 11]
 Let $f=\mathcal{I}(F)\in\mathcal{S}^1(\Omega_D)$ with $D\subset\mathbb{C}^+$ (so that $\beta>0$). Then, on $\Omega_D\setminus\mathbb{R}$, the following equality holds true
 \begin{equation*}
     d_{sl}f=\frac{\partial f}{\partial \alpha}d\alpha+\frac{\partial f}{\partial \beta}d\beta.
    \end{equation*}
\end{proposition}

It is clear from the definition that, if we choose the usual coordinate system, where $x=\alpha+I\beta$ with $\beta>0$, 
then $d_{sl}x=d\alpha +Id\beta$ and $d_{sl}x^c=d\alpha-Id\beta$.
We can now state the following theorem.

\begin{theorem}[\cite{altavilladiff}, Theorem 7]
 Let $f\in\mathcal{S}^1(\Omega_D)$. Then the following equality holds:
 \begin{equation*}
  d_{sl}x\frac{\partial f}{\partial x}(x)+d_{sl}x^c\frac{\partial f}{\partial x^c}(x)=d_{sl}f(x),\quad\forall x\in\Omega_D\setminus\mathbb{R}.
 \end{equation*}
\end{theorem}

We have then the obvious corollary:
\begin{corollary}
 Let $f\in\mathcal{SR}(\Omega_D)$. Then the following equality holds:
 \begin{equation*}
  d_{sl}x\frac{\partial f}{\partial x}(x)=d_{sl}f(x),\quad\forall x\in\Omega_D\setminus\mathbb{R}.
 \end{equation*}
\end{corollary}

Some important classes of slice regular functions are now introduced.

\begin{definition}[Slice constant function]
Let $\Omega_D$ be a connected circular domain and
 let $f=\mathcal{I}(F)\in\mathcal{S}(\Omega_{D})$. $f$ is called \textit{slice constant} if the stem function $F$ is locally constant on $D$.
\end{definition}


\begin{proposition}[\cite{altavilla}, Proposition 3.3 and Theorem 3.4]\label{slccnst}
Let $f\in\mathcal{S}(\Omega_{D})$ be a slice function. If $f$ is slice constant then it is slice regular. 
Moreover $f$ is slice constant if and only if
\begin{equation*}
\frac{\partial f}{\partial x}\equiv 0.
\end{equation*}
\end{proposition}

\begin{remark}
The previous proposition tells that if we have a slice constant function $f\in\mathcal{SR}(\Omega_{D})$ over a connected
circular domain $\Omega_D$, then, given $J\in\mathbb{S}$, if $x\in D_{J}^+\setminus \mathbb{R}$
\begin{equation*}
f(x)=a+ Jb=a+\frac{Im(x)}{||Im(x)||}b,\quad a,b\in\mathbb{H}.
\end{equation*}
\end{remark}

\begin{proposition}\label{propslcnst}
Let $\Omega_D$ be a connected circular domain.
Let $g:\Omega_D\rightarrow \mathbb{H}$ be a slice function. $g$ is slice constant if and only if
given any fixed $J\in\mathbb{S}$, $g|_{\Omega_D\setminus\mathbb{R}}$ is a
 linear combination, with right quaternionic coefficients, of the two functions 
 $g_+,g_{-}:\mathbb{H}\setminus\mathbb{R}\rightarrow \mathbb{H}$ defined by $g_\pm(\alpha+I\beta )=1\pm IJ$.
\end{proposition}

\begin{proof}
 Thanks to Theorem~\ref{slccnst}, any linear combination of the two functions $g_+,g_{-}$ is slice constant since their slice derivative is everywhere zero.
%
Vice versa, let $g=\mathcal{I}(g_1+\sqrt{-1}g_2)$ be a slice constant function, with $g_{1},g_{2}\in\mathbb{H}$. Thanks to the Representation Formula~\eqref{repreform}, for any $J\in\mathbb{S}$ we have $g(\alpha+I\beta)=[(1-IJ)(g_1+Jg_2)+(1+IJ)(g_1-Jg_2)]/2$ and
 the last is equal to $g(\alpha+I\beta)=[g_{-}(g_1+Jg_2)+g_{+}(g_1-Jg_2)]/2$.
 
\end{proof}

Now we will introduce the set of slice regular function that are affine slice by slice. This notion will be useful in a next result.
\begin{definition}[Slice affine functions]
 Let $f:\Omega_D\rightarrow\mathbb{H}$ be a slice regular function. $f$ is called \textit{slice affine} if its slice derivative is a slice constant function.
\end{definition}

\begin{proposition}
Let $f:\Omega_D\rightarrow \mathbb{H}$ be a slice function. $f$ is slice affine if and only if
 given any fixed $J\in\mathbb{S}$, $f|_{\Omega_D\setminus\mathbb{R}}$ is a
 linear combination, with right quaternionic coefficient, of the four functions $f_+,f_{-},g_+,g_{-}:\mathbb{H}\setminus\mathbb{R}\rightarrow \mathbb{H}$,
 where $g_+,g_{-}$ are the one defined in Proposition~\ref{propslcnst} and $f_\pm(\alpha+I\beta)=(\alpha+I\beta)g_\pm(\alpha+I\beta)$.
\end{proposition}

\begin{proof}
 If $f$ is a linear combination of $f_+,f_{-}$ and $g_+,g_{-}$ then it is obviously a slice affine function. 
 Vice versa, since $\partial f/\partial x$ is a slice constant function, then, in the language of slice forms
 \begin{equation*}
  d_{sl}f=d_{sl}x\frac{\partial f}{\partial x}=d_{sl}xg(x),
 \end{equation*}
 with $g=\mathcal{I}(g_1+\sqrt{-1}g_2)$ a slice constant function.
 The previous equality, using the definition of slice form, is equivalent to the following one
  \begin{equation*}
  \frac{\partial F_1}{\partial \alpha}d\alpha+\frac{\partial F_1}{\partial\beta}d\beta+
  I\left(\frac{\partial F_2}{\partial\alpha}d\alpha+\frac{\partial F_2}{\partial \beta}d\beta\right)=g_1d\alpha-g_2d\beta+I(g_2d\alpha+g_1d\beta),
 \end{equation*}
 that implies $F_1=g_1\alpha-g_2\beta+q_1$ and $F_2=g_2\alpha+g_1\beta+q_2$, for some couple $q_1,q_2$ of quaternions.
 But then, by applying the Representation Formula~\eqref{repreform}, and using the same argument as in the proof of Proposition~\ref{propslcnst} we obtain,
 \begin{equation*}
  \begin{array}{rcl}
   f(\alpha+I\beta) & = & =g_1\alpha-g_2\beta+I(g_2\alpha+g_1\beta)+q_1+Iq_2\\
   & = & \alpha(g_1+Ig_2)+I\beta(g_1+Ig_2)+q_1+Iq_2\\
   & = & (\alpha+I\beta)(g_1+Ig_2)+q_1+Iq_2\\
   & = & (\alpha+I\beta)[(1-IJ)(g_1+Jg_2)+(1+IJ)(g_1-Jg_2)]/2+\\
   & & +[(1-IJ)(q_1+Jq_2)+(1+IJ)(q_1-Jq_2)]/2.
  \end{array}
 \end{equation*}
\end{proof}

\begin{remark}\label{extend}
 The set of slice constant functions contains the set of constant functions and the condition for a slice constant function $g=g_+q_++g_-q_-$ to 
 be extended to $\mathbb{R}$ is 
 that $q_+=q_-$ (i.e.: $g$ is a constant function). Analogously, a slice affine function $f=f_+q_{1+}+f_-q_{1-}+g_+q_{0+}+g_-q_{0-}$ extends to the real 
 line if and only if $q_{1+}=q_{1-}$ and $q_{0+}=q_{0-}$ (i.e.: $f=xa+b$ is a $\mathbb{H}$-affine function). For 
 slice constant functions the assertion is trivial while for slice affine functions it requires a simple consideration regarding the limit of a slice affine function for $\beta$ that 
 approach $0$ when $\beta$ is lower or greater than zero. In formula, the previous condition is the following one:
 \begin{equation*}
  \lim_{\underset{\alpha+I\beta\in\mathbb{C}_I^+}{\beta\to 0}} f(\alpha+I\beta)=  \lim_{\underset{\alpha+I\beta\in\mathbb{C}_I^-}{\beta \to 0}} f(\alpha+I\beta).
 \end{equation*}

\end{remark}

\begin{remark}
One can define, in general, the class of ``slice polynomial'' functions as the set of slice 
regular functions such that the $n^{th}$ slice derivative vanishes for some $n$.
This can be actually a useful notion in view of some researches regarding the number of
counterimages of a slice regular function defined over a domain without real points.
Anyway this  theme is not explored in this paper.
\end{remark}

%

In the next part of this section we will recall some theorems regarding the nature of
the real differential of a slice regular function. These results can be found in~\cite{stoppato, gensalsto} and 
their generalization in~\cite{altavilladiff}.
Firstly we will expose a representation of the differential. As we said before, given a slice regular function,
its spherical derivative and its slice derivative control the variation of the functions along 
spheres $\mathbb{S}_{x}$ and slices $\mathbb{C}_{I}$, respectively. This is in fact the content
of the following theorem.
\begin{theorem}[\cite{altavilladiff,stoppato}]\label{reprediff}
 Let $f\in\mathcal{SR}(\Omega_D)$ and let $(df)_x$ denote the real differential of $f$ at $x=\alpha+I\beta\in\Omega_D\setminus\mathbb{R}$. If we identify $T_x\Omega_{D}$ with
 $\mathbb{H}=\mathbb{C}_I\oplus \mathbb{C}_I^\bot$, then for all $v_1\in\mathbb{C}_I$ and $v_2\in\mathbb{C}_I^\bot$,
 \begin{equation*}
  (df)_x(v_1+v_2)=v_1\frac{\partial f}{\partial x}(x)+v_2\partial_sf(x).
 \end{equation*}
  If $\alpha\in\Omega_D\cap\mathbb{R}$ then, the previous formula becomes the following one
  \begin{equation*}
  (df)_\alpha(v)=v\frac{\partial f}{\partial x}(\alpha)=v\partial_sf(\alpha).
 \end{equation*}
\end{theorem}

\begin{proposition}[\cite{gensalsto}, Proposition 3.3]
 Let $f\in\mathcal{SR}(\Omega_D)$ and $x_0=\alpha+J\beta\in\Omega_D\setminus\mathbb{R}$.
 \begin{itemize}
  \item If $\partial_sf(x_0)=0$ then:
  \begin{itemize}
   \item $df_{x_0}$ has rank 2 if $\frac{\partial f}{\partial x}(x_0)\neq 0$;
   \item $df_{x_0}$ has rank 0 if $\frac{\partial f}{\partial x}(x_0)= 0$.
  \end{itemize}
  \item If $\partial_sf(x_0)\neq0$ then $df_{x_0}$ is not invertible at $x_0$ if and only if $\frac{\partial f}{\partial x}(x_0)(\partial_sf(x_0))^{-1}$ belongs
  to $\mathbb{C}_J^\bot$.
 \end{itemize}
Let now $\alpha\in\Omega_D\cap \mathbb{R}$. $df_{x_0}$ is invertible at $\alpha$ if and only if its rank is not 0 at $x_0=\alpha +J\beta$. This happens if and only if 
$\partial_sf(x_0)=\frac{\partial f}{\partial x}(x_{0})\neq 0$.
\end{proposition}

\begin{definition}[Singular set]
 Let $f:\Omega\rightarrow\mathbb{H}$ be any quaternionic function of quaternionic variable. We define the \textit{singular set} of $f$ as
 \begin{equation*}
  N_f:=\{x\in\Omega\,|\,df\mbox{ is not invertible at }x\}.
 \end{equation*}
\end{definition}

Given a slice regular function $f$ that is not slice constant, then its singular set $N_{f}$ is closed
with empty interior; moreover if $f$ is injective then it spherical and  slice derivatives are both nonzero.
With some other information it is possible to prove the following theorem.
\begin{theorem}[\cite{gensalsto,altavilladiff}]\label{injective}
Let $f$ be an injective slice regular function, then $N_{f}=\emptyset$.
\end{theorem}

%

\section{Twistor lift}\label{twlift}
Starting from Theorem~\ref{parabola} (or equivalently from Theorem~\ref{classification}, part (1)), 
we know that, up to sign, on $\mathbb{H}\setminus\mathbb{R}$ it is possible to define only one
non-constant OCS. This OCS can be defined as follows.

\begin{definition}[Slice complex structure]
 Let $p=\alpha+I_p\beta\in X=\mathbb{H}\setminus\mathbb{R}$ with $\beta>0$, and let us identify $T_pX\simeq \mathbb{H}$. We define the following 
 OCS over $X$:
 \begin{equation*}
  \mathbb{J}_{p}v=\frac{Im(p)}{||Im(p)||} v = I_p v,
 \end{equation*}
where $v$ is a tangent vector to $X$ in $p$ and $I_{p}v$
denotes the quaternionic multiplication between $I_{p}$ and $v$.
\end{definition}
Later we will describe the algebraic surface in $\mathbb{CP}^{3}$ arising from this OCS,
but now let us go back again to quaternionic functions.
The theory regarding the relation between slice regular functions and twistor geometry
starts thanks to Theorems~\ref{reprediff} and~\ref{injective} that extend two results
proved, respectively in~\cite{stoppato} and~\cite{gensalsto}.

Given an injective slice regular function $f:\Omega_D\rightarrow \mathbb{H}$  we define the pushforward of 
$\mathbb{J}$ via $f$ on $f(\Omega_D\setminus \mathbb{R})$ as:
\begin{equation*}
 \mathbb{J}^f:=(df)\mathbb{J}(df)^{-1},
\end{equation*}
for any $v\in T_{f(p)}f(\Omega_{D}\setminus\mathbb{R})\simeq\mathbb{H}$.

 The following theorem explains the action of the push-forward of $\mathbb{J}$ via a  slice regular function.
\begin{theorem}
Let $f:\Omega_{D}\rightarrow\mathbb{H}$ be an injective slice regular function and $p=\alpha+I_p\beta\in\Omega_D$. Then
\begin{equation*}
\mathbb{J}^f_{f(p)}v=\frac{Im(p)}{||Im(p)||} v=I_pv.
\end{equation*}
Moreover $\mathbb{J}^f$ is an OCS on the image of $f$.
\end{theorem}

\begin{proof}
The theorem can be proved as in~\cite{gensalsto}, but we will write again the proof using the representation given in Theorem~\ref{reprediff} of the real differential of
a slice regular function. The thesis follows thanks to the next computations. Let $v$ be a tangent vector to $f(\Omega_D\setminus\mathbb{R})$ in $f(x)$
\begin{equation*}
 \mathbb{J}_{f(x)}^fv=(df)_x\mathbb{J}_x(df)_{f(x)}^{-1}v.
\end{equation*}
Setting $(df)_{f(x)}^{-1}v=w$ and denoting by $w_\top$ and $w_\bot$, respectively, the tangential and orthogonal part of $w$ with respect to $\mathbb{C}_{I_x}$, we obtain,
\begin{equation*}
 \begin{array}{rcl}
  (df)_x\mathbb{J}_x(df)_{f(x)}^{-1}v & = & (df)_x\mathbb{J}_x w=(df)_x I_x w\\
  & = & I_xw_\top \displaystyle\frac{\partial f}{\partial x}(x)+I_x w_\bot\partial_s f(x)\\
  & = & I_x(df)_x w=I_xv.
 \end{array}
\end{equation*}

For the second part of the theorem we refer again to~\cite{gensalsto}, anyway, it is enough to
compute the quantity $g_{Eucl}(\mathbb{J}X,\mathbb{J}Y)$, pointwise. 
\end{proof}
%
%

In the next pages we will recover, in the context of slice regular functions defined on domains
without real point, the twistor theory introduced in~\cite{gensalsto}.
First of all we need to introduce coordinates for the sphere $\mathbb{S}$ of imaginary units. For this purpose we
will follow the construction in~\cite[Section 4]{gensalsto}. 
For any $q=\alpha+I\beta$, $\beta>0$, if $I=ai+bj+ck$, we can write
$$
q=\alpha+I\beta=\alpha+Q_{u}^{-1}iQ_{u}\beta,
$$
where $Q_{u}=1+uj$ and $u=-i\frac{b+ic}{1+a}$.

%
%

We embed now $\mathbb{H}\setminus\mathbb{R}$ in $\mathbb{HP}^1$ via the function $q\rightarrow [1,q]$. Given $q=\alpha+I\beta\in\mathbb{H}\setminus\mathbb{R}$, if $u$ is defined as above and $v=\alpha+i\beta$, such an embedding can be 
viewed, also, in the following way:
\begin{equation*}
 \begin{array}{rcl}
  [1,q] & = & [1,Q_u^{-1}vQ_u]=[Q_u,vQ_u]\\
  & = & [1+uj,v+vuj]=\pi[1,u,v,uv],
 \end{array}
\end{equation*}
and so, we have obtained, as in~\cite{gensalsto}, the following proposition.
\begin{proposition}
 The complex manifold $(\mathbb{H}\setminus\mathbb{R},\mathbb{J})$ is biholomorphic to the open subset $\mathcal{Q}^+$ of the quadric
 \begin{equation}\label{quadric}
  \mathcal{Q}=\{[X_0,X_1,X_2,X_3]\in\mathbb{CP}^3\,\mid\,X_0X_3=X_1X_2\},
 \end{equation}
such that at least one of the following conditions is satisfied:
\begin{itemize}
 \item $X_0\neq 0$ and $X_2/X_0\in\mathbb{C}^+$,
 \item $X_1\neq 0$ and $X_3/X_1\in\mathbb{C}^+$.
\end{itemize}
\end{proposition}
The quadric $\mathcal{Q}$ is biholomorphic to $\mathbb{CP}^1\times\mathbb{CP}^1$, while $\mathcal{Q}^+$ is biholomorphic to $\mathbb{CP}^1\times\mathbb{C}^+$.

Now we have all the ingredients to state the following theorem which generalizes~\cite[Theorem 5.3]{gensalsto}.
\begin{theorem}\label{thmlift}
 Let $D$ be a domain of $\mathbb{C}$ and $\Omega_D\subset\mathbb{H}$ its circularization. Let $f:\Omega_D\rightarrow\mathbb{H}$  be
 any slice function. Then $f$ admits a twistor lift to $\mathcal{O}=\pi^{-1}(\Omega_D\setminus\mathbb{R})\cap \mathcal{Q}^+$, i.e.: there exists a 
 function $\tilde{f}:\mathcal{O}\rightarrow \mathbb{CP}^3$, such that $\pi\circ \tilde{f}=f\circ\pi$. 
 Moreover $f$ is slice regular if and only if $\tilde{f}$ is a holomorphic map.
\end{theorem}

As we said, this theorem was already proven in~\cite{gensalsto}, when the domain $D$ has nonempty intersection with the real line and  the function $f$ is regular. 
Our proof includes also the case in which $f$ does not extends to the real line and it is not regular, so it is more general. 
To add this extension we will use the previously described formalism of stem
functions to which we add this trivial lemma that is a consequence of~\cite[Lemma 6.11]{ghilmorper}.

\begin{lemma}
Let $f=\mathcal{I}(F):\Omega_{D}\rightarrow \mathbb{H}$ be a slice function induced by the 
stem function $F:D\rightarrow \mathbb{H}_{\mathbb{C}}$. Then, for each couple 
$I,J\in\mathbb{S}$ such that $I\bot J$, there exist two stem functions 
$F^{\top},F^{\bot}:D\rightarrow\mathbb{C}_{I}\otimes_{\mathbb{R}}\mathbb{C}$, such that
$f=f^{\top}+f^{\bot}J$ with $f^{\top}=\mathcal{I}(F^{\top})$, while $f^{\bot}=\mathcal{I}(F^{\bot})$.
\end{lemma}

Now we pass to the proof of Theorem~\ref{thmlift}.

\begin{proof}
Since $f$ is a slice function, then it is induced by a stem function $F:D\rightarrow \mathbb{H}_{\mathbb{C}}$ such that, for $q=\alpha+I\beta\in\Omega_{D}$,
\begin{equation*}
f(q)=f(\alpha+I\beta)=f(\alpha+Q_u^{-1}iQ_u\beta)=F_1(\alpha+i\beta)+Q_u^{-1}iQ_uF_2(\alpha+i\beta).
\end{equation*}
Thanks to the previous lemma $f$ can be written also as $f=f^{\top}+f^{\bot}j$, with $f^{\top}=\mathcal{I}(F^{\top})$, $f^{\bot}=\mathcal{I}(F^{\bot})$, $F^{\top},F^{\bot}:D\rightarrow \mathbb{C}_{i}\otimes_{\mathbb{R}}\mathbb{C}$.
Now, each stem function splits into two components, 
$F^{\top}=F_{1}^{\top}+\sqrt{-1}F_{2}^{\top}$ and $F^{\bot}=F_{1}^{\bot}+\sqrt{-1}F_{2}^{\bot}$, and we define, for  
$i\in\mathbb{S}$,
$F_{i}^{\top}=p_{i}\circ F^{\top}$ and $F_{i}^{\bot}=p_{i}\circ F^{\bot}$, where $p_{i}$ is the map that sends $\sqrt{-1}$ to
$i$ (i.e. if $w=x+\sqrt{-1}y\in\mathbb{H}_{\mathbb{C}}$, then $p_i(w)=x+iy$). To summarize we have the following diagram
$$
\begindc{\commdiag}[50]
\obj(0,10)[aa]{$D$}
\obj(20,10)[bb]{$\mathbb{C}_{i}\otimes_{\mathbb{R}}\mathbb{C}$}
\obj(20,0)[dd]{$\mathbb{C}_{i}$}
\mor{aa}{bb}{$F^{\top},F^{\bot}$}
\mor{bb}{dd}{$p_{i}$}
\mor{aa}{dd}{$F_{i}^{\top},F_{i}^{\top}$}[-1,0]
\enddc
$$
Letting finally $q=\alpha+I\beta$ and $v=\alpha+i\beta$ and recalling that $Q_{u}=1+uj$, we can compute,
 \begin{equation*}
  \begin{array}{rcl}
   [1,f(q)] & = & [1,f(Q_u^{-1}(\alpha+i\beta)Q_u)]\\
   & = & [1,f(\alpha+Q_u^{-1}iQ_u\beta)]\\
   & = & [1,F^{\top}_{1}(v)+Q_u^{-1}iQ_u F^{\top}_{2}(v)+F^{\bot}_{1}(v)j+Q_u^{-1}iQ_uF^{\bot}_{2}(v)j]\\
   & = & [Q_u,F^{\top}_{1}+ujF^{\top}_{1}+iF^{\top}_{2}+uijF^{\top}_{2}+F^{\bot}_{1}j+ujF^{\bot}_{1}j+iF^{\bot}_{2}j+uijF^{\bot}_{2}j],   
   \end{array}
   \end{equation*}
where in the last equality we have omitted the variable $v$. 
Now, for any $w\in\mathbb{C}_{i}$,
we have that $jw=w^{c}j$ and $jwj=-w^{c}$ and so, identifying $\mathbb{C}_{i}$ with $\mathbb{C}$,
   
\begin{multline*}
    [Q_u,F^{\top}_{1}+ujF^{\top}_{1}+iF^{\top}_{2}+uijF^{\top}_{2}+F^{\bot}_{1}j+ujF^{\bot}_{1}j+iF^{\bot}_{2}j+uijF^{\bot}_{2}j]  =\\   [Q_u,F^{\top}_{1}+uF^{\top c}_{1}j+iF^{\top}_{2}+uiF^{\top c}_{2}j+F^{\bot}_{1}j-uF^{\bot c}_{1}+iF^{\bot}_{2}j-uiF^{\bot c}_{2}]=\\
=    [Q_u,F^{\top}_{1}+iF^{\top}_{2}+(F^{\bot}_{1}+iF^{\bot}_{2})j+u((F^{\top c}_{1}+iF^{\top c}_{2})j-(F^{\bot c}_{1}+iF^{\bot c}_{2}))].
\end{multline*}
 We finally obtain the coordinates of the lift:
   \begin{equation}\label{lift}
   \tilde{f}[1,u,v,uv]=[1,u,p_{i}\circ F^{\top}(v)-u(p_{i}\circ F^{\bot c}(v)),p_{i}\circ F^{\bot} (v)+u(p_{i}\circ F^{\top c}(v))].
   \end{equation}
  But now, recalling that $F^{\top}, F^{\bot}$ are holomorphic stem functions, then, we have that $f$ is slice regular if and only if $\tilde{f}$ is a holomorphic map.
\end{proof}

\begin{remark}\label{usuallift}
Starting with a slice regular function $f$, one can repeat the computations in the following way
 \begin{equation*}
  \begin{array}{rcl}
   [1,f(q)] & = & [1,f(Q_u^{-1}(\alpha+i\beta)Q_u)]\\
   & = & [1,f(\alpha+Q_u^{-1}(i)Q_u\beta)]\\
   & = & [1,F_1(\alpha+i\beta)+Q_u^{-1}(i)Q_uF_2(\alpha+i\beta)]\\
   & = & [Q_u,Q_uF_1(\alpha+i\beta)+iQ_uF_2(\alpha+i\beta)]\\
   & = & [1+uj,(1+uj)F_1(\alpha+i\beta)+i(1+uj)F_2(\alpha+i\beta)]\\
   & = & [1+uj,f(\alpha+i\beta)+ujf(\alpha-i\beta)]\\
   & = & [1+uj,f(v)+ujf(\bar v)].
  \end{array}
 \end{equation*}
At this point, using the splitting in Lemma~\ref{splitting}, we can  write $f_i(v)=G(v)+H(v)j$, where
$G,H: D_i\rightarrow \mathbb{C}_i$ are holomorphic functions. 
Now, if $D_{i}\cap\mathbb{R}=\emptyset$, then the behaviours of $G$ and $H$ over
$D_{i}^{+}$ and $D_{i}^{-}$ are, in general, unrelated. We write then
\begin{equation*}
G(v):=\begin{cases}
g(v) & v\in D_{i}^{+}\\
\overline{\hat{g}(\bar v)} & v\in D_{i}^{-}
\end{cases},\quad H(v):=\begin{cases}
h(v) & v\in D_{i}^{+}\\
\overline{\hat{h}(\bar v)} & v\in D_{i}^{-},
\end{cases}
\end{equation*}
where $g,\hat{g},h,\hat{h}$ are holomorphic functions defined on $D_{i}^{+}$.
This is done because, in the lift, the variable $v$ belongs to $\mathbb{C}^{+}$ and
can be done because $D$ is symmetric with respect to the real axis.
Note that, since $\hat{g}$ and $\hat{h}$ are holomorphic functions, then the two functions
$v\mapsto \overline{\hat{g}(\bar v)}$, $v\mapsto\overline{\hat{h}(\bar v)}$ are holomorphic as well.
%
This particular choice is made in order to let the result compatible with the one in~\cite{gensalsto}.
Coming back to our computations we get,
\begin{equation*}
 ujf(\alpha-i\beta)=u(\overline{G(\bar v)}j-\overline{H(\bar v)}),
\end{equation*}
and since $v\in\mathbb{C}^{+}$, then $\overline{G(\bar v)}=\overline{\overline{\hat{g}(\overline{\overline{ v}})}}=\hat{g}(v)$ (and analogously for $H$), hence,
 \begin{equation}\label{twistorsplitting}
  \begin{array}{rcl}
   [1+uj,f(v)+ujf(\bar v)] & = & [1+uj,g(v)+h(v)j-u\hat{h}(v)+u\hat{g}(v)j]\\
   & = & \pi[1,u,g(v)-u\hat{h}(v),h(v)+u\hat{g}(v)],
  \end{array}
 \end{equation}
and so the lift coincide with the one computed in~\cite{gensalsto}. 

Finally, observe that if a slice regular function is defined over a domain which intersects the real line,
then (as implicitly stated in~\cite{genstostru}), $\hat{g}=\overline{g(\bar v)}$ and analogously for $h$.
\end{remark}

\begin{remark}
 It will be useful to notice that the twistor lift of a slice regular 
 function is always a rational map over its image.
\end{remark}

Thanks to the Representation Formula for slice functions, exhibiting a slice function
is equivalent to exhibit its defining stem function or its splitting over a complex plane $\mathbb{C}_I$ for some $I\in\mathbb{S}$.
With this in mind, in the next proofs and constructions we will define $G$ and $H$ starting from Equation~\eqref{twistorsplitting}.
In particular, given a slice regular function $f:\Omega_D\rightarrow \mathbb{H}$ which splits over $D_i$ as $f(v)=G(v)+H(v)j$, it holds:
\begin{equation*}
 \begin{array}{rcl}
  f(\alpha+I\beta) & = & \frac{1}{2}[f(v)+f(\bar v)-Ii(f(v)-f(\bar v))]\\
  & = & \frac{1}{2}[G(v)+H(v)j+G(\bar v)+H(\bar v)j-Ii(G(v)+H(v)j-G(\bar v)-H(\bar v)j)]\\
  & = & \frac{1}{2}[(1-Ii)(g(v)+h(v)j)+(1+Ii)(\overline{\hat{g}(v)}+\overline{\hat{h}(v)}j)],
 \end{array}
\end{equation*}
where $\alpha+I\beta\in\Omega_D\setminus\mathbb{R}$ and $v=\alpha+i\beta$, with $\beta>0$.

Given a slice regular function $f$ we will say that its twistor lift
$\tilde{f}$ lies on a certain variety $\mathcal{S}$ if the image of 
$\tilde{f}$ is contained in 
$\mathcal{S}$.
\subsection{Planes}

Here we will show that, given a hyperplane in $\mathbb{CP}^{3}$, then, the only
non-constant slice regular functions  that arise in our constructions are functions that do not extend
to the real line.
Take in fact a generic hyperplane given by the equation:
\begin{equation*}
c_{0}X_{0}+c_{1}X_{1}+c_{2}X_{2}+c_{3}X_{3}=0.
\end{equation*}
Substituting the coordinates in Equation~\eqref{twistorsplitting} in the previous equation
we get
\begin{equation*}
c_{0}+c_{1}u+c_{2}(g(v)-u\hat{h}(v))+c_{3}(h(v)+u\hat{g}(v))=0.
\end{equation*}
The left hand side of the last equation is, of course, a linear polynomial in $u$, so,
the equality holds if and only if the next system is satisfied,
\begin{equation}\label{eqplane}
\begin{cases}
c_{0}+c_{2}g(v)+c_{3}h(v)=0\\
c_{1}-c_{2}\hat{h}(v)+c_{3}\hat{g}(v)=0.
\end{cases}
\end{equation}
Due to the nature of the lift we have to suppose that at least one between $c_{2}$ and 
$c_{3}$ is different from zero. Say, then $c_{3}\neq 0$ (the case $c_{2}\neq 0$ is
obviously symmetric). Then, the last system becomes, 
\begin{equation*}
\begin{cases}
h(v)=-c_{3}^{-1}(c_{0}+c_{2}g(v))\\
\hat{g}(v)=-c_{3}^{-1}(c_{1}-c_{2}\hat{h}(v)).
\end{cases}
\end{equation*}

Using now the Representation Formula~\eqref{repreform}, we can define the following slice regular
function $f:\mathbb{H}\setminus\mathbb{R}\rightarrow\mathbb{H}$, as
\begin{equation*}
f(\alpha+I\beta)=\frac{(1-Ii)}{2}(g(v)-c_{3}^{-1}(c_{0}+c_{2}g(v))j)+\frac{(1+Ii)}{2}(
-\overline{c_{3}^{-1}(c_{1}-c_{2}\hat{h}( v))}+\overline{\hat{h}( v)}j).
\end{equation*}

If this function extends to $\mathbb{R}$, then $\hat{g}(v)=\overline{g(\bar v)}$
and $\hat{h}(v)=\overline{h(\bar v)}$. But if this is true, then from the system in Equation~\ref{eqplane} we get,
\begin{equation*}
\hat{h}(v)=\overline{-c_{3}^{-1}c_{0}}+\overline{-c_{3}^{-1}c_{2}}\hat{g}(v),
\end{equation*}
and substituting this in the second equation we obtain that $\hat{g}$ is equal to some
constant complex number. Hence, the only way to obtain a slice regular function
that extends to the real line is to take a constant function.

\subsection{Quadrics}

Thanks only to the general shape of the lift given in Equation~\eqref{lift},
we are able to prove the following result.

\begin{theorem}\label{realquad}
 Let $f:\mathbb{H}\setminus\mathbb{R}\rightarrow \mathbb{H}$ be a slice regular function. 
 Then its twistor lift lies over the quadric in Formula~\eqref{quadric} if and only if $f$ is a real slice function.
\end{theorem}
\begin{proof}
 Since the parameterization of the lift $\tilde{f}$ is given by Equation~\eqref{lift}, then the condition of
 lying on the quadric given in Formula~\eqref{quadric} is encoded by the following system of equations

 \begin{numcases}{}
\label{real1} p_{i}\circ F^{\bot c}=0=p_{i}\circ F^{\bot}\\
\label{real2} p_{i}\circ F^{\top c}=p_{i}\circ F^{\top},
\end{numcases}
and so the slice regular function $f$ with lifting equal to $\tilde{f}$ can be constructed, thanks to Equation~\eqref{real1}, to be equal to,
\begin{equation*}
 f(\alpha+I\beta)=f^{\top}(\alpha+I\beta)=F^{\top}_{1}(\alpha+i\beta)+IF^{\top}_{2}(\alpha+i\beta).
\end{equation*}
But, thanks to Equation~\eqref{real2} we have that 
\begin{equation*}
F^{\top}_{1}(\alpha+i\beta)+IF^{\top}_{2}(\alpha+i\beta)=F^{\top c}_{1}(\alpha+i\beta)+IF^{\top c}_{2}(\alpha+i\beta),
\end{equation*}
which implies that both $F^{\top}_{1}, F^{\top}_{2}$ are real functions and so $f$ is real. 

The converse is trivial.
\end{proof}

Now, the next result states that every non-singular quadric in the previous classification (see Theorem~\ref{classification}), can be reached by the lift of a slice regular function.

\begin{theorem}\label{thmquadric}
 For any non-singular quadric in the classification of Theorem~\ref{classification} there is an  equivalent one $\mathcal{Q}$ for which there exists a slice regular function 
 $f$ defined on a dense subset of $\mathbb{H}\setminus\mathbb{R}$, such that its twistor lift lies in $\mathcal{Q}$.
\end{theorem}
\begin{proof}
 For all the  cases we
 will show the thesis by exhibiting the splitting of $f$.
 \begin{enumerate}
 \item If $\mathcal{Q}$ is given as in Equation~\eqref{quaddiag}, then it translates in 
 set of solutions of
 \begin{equation*}
 e^{\lambda+i\nu}+e^{-\lambda+i\nu}u^{2}+e^{\mu-i\nu}(g(v)-u\hat{h}(v))^{2}+e^{-\mu-i\nu}
 (h(v)+u\hat{g}(v))^{2}=0.
 \end{equation*}
Writing the previous equation as a polynomial in $u$ and imposing the vanishing of
the coefficients we obtain the following system
\begin{equation*}
\begin{cases}
e^{\lambda+i\nu}+e^{\mu-i\nu}g^{2}+e^{-\mu-i\nu}h^{2}=0\\
-e^{\mu}g\hat{h}+e^{-\mu}h\hat{g}=0\\
e^{-\lambda+i\nu}+e^{\mu-i\nu}\hat{h}^{2}+e^{-\mu-i\nu}\hat{g}^{2}=0
\end{cases}
\end{equation*}
From the first and the last equations we obtain
\begin{equation*}
h^{2}=-e^{\mu+i\nu}(e^{\lambda+i\nu}+e^{\mu-i\nu}g^{2}),\quad \hat{h}^{2}=-e^{-\mu+i\nu}(e^{-\lambda+i\nu}+e^{-\mu-i\nu}\hat{g}^{2}).
\end{equation*}
Take now the square of  second equation and substitute the values of $h^{2}$ and $\hat{h}^{2}$:
\begin{equation*}
e^{\mu}g^{2}(e^{-\lambda+i\nu}+e^{\mu-i\nu}\hat{g}^{2})=e^{-\mu}\hat{g}^{2}(e^{\lambda+i\nu}+e^{\mu-i\nu}g^{2}),
\end{equation*}
that is
\begin{equation*}
\hat{g}=\pm e^{\mu-\nu}g.
\end{equation*}
Taking now, for instance, $g(v)=v$, $\hat{g}(v)=e^{\mu-\nu}v$, $h=i(e^{\mu+i\nu}(e^{\lambda+i\nu}+e^{\mu-i\nu} g^{2}))^{1/2}$ and  $\hat{h}=i(e^{-\mu+i\nu}(e^{-\lambda+i\nu}+e^{-\mu-i\nu}\hat{g}^{2}))^{1/2}$, we get the thesis in the first case.

%

\item The last case is when $\mathcal{Q}$ is the zero locus of the polynomial in Equation~\eqref{quadnondiag} with $k\in[0,1)$.
Imposing then the usual equations we obtain that $g,h:\mathbb{C}_i\setminus\mathbb{R}\rightarrow\mathbb{C}$
and $\hat{g},\hat{h}:\mathbb{C}_i^-\setminus\mathbb{R}\rightarrow\mathbb{C}$
can be chosen as

\begin{equation*}
 g(v)=-\hat{g}(v)=v,\quad h(v)=2i+v/2,\quad  \hat{h}(v)=2i-v/2.
\end{equation*}
 \end{enumerate}
It is now a matter of computation, using the Representation Formula, to write the slice regular functions defined by the previous three cases. 
\end{proof}
%
Given the parameterizations contained in the previous proof, in the next section we will compute,
for each quadric in the classification of Theorem~\ref{classification}, the sets of points where their possible twistor lines lie.
%

In the next theorem we will show that, up to projective transformations, the only non-singular algebraic surface that can be parametrized by the twistor lift of a slice regular function, is exactly $\mathcal{Q}$.
Some suspects that a result of this kind must hold came from the fact that there are not
\textit{dominant rational maps}\footnote{Meaning a rational map with dense image.} from $\mathcal{Q}$ to any smooth variety of degree $d\geq 4$.
In fact,  any smooth quadric in $\mathbb{CP}^3$ is projectively isomorphic to $\mathcal{Q}$ (see, for instance,~\cite[Section 4]{griffiths}).
  Now, if $X\rightarrow Y$ is a 
 dominant rational map between non-singular varieties in $\mathbb{CP}^{3}$, then
  $dim H^0(Y,K_Y)\leq dim H^0(X,K_X)$, where $K_X$ and $K_{Y}$ stand for the canonical bundles of the subscript variety (see~\cite[Chapter 2, Section 8]{hartshorne}).
But   $dim H^0(\mathcal{S},K_\mathcal{S})$
 is  greater or equal to 1 when the degree of $\mathcal{S}$ is greater or equal to 4 and it is 0 when $d=2,3$. 
%
  
To be more precise we have the following.
\begin{theorem}

 Let $\mathcal{S}$ be a non-singular algebraic surface of degree $d\geq 2$ in $\mathbb{CP}^3$ and let $\tilde{f}:\mathcal{Q}^+\rightarrow \mathcal{S}$ be
 the twistor lift of a slice regular function and such that $\tilde{f}(\mathcal{Q}^+)$ is open in $\mathcal{S}$. Then $\mathcal{S}$
 is projectively equivalent to $\mathcal{Q}$.

\end{theorem}

\begin{proof}
%
Observe that for each fixed $v_{0}$ in
 $\mathbb{CP}^{1}$, the twistor lift $\tilde{f}$ of a generic slice regular function $f$, contains
  the whole line $l_{v_{0}}:\mathbb{CP}^{1}\rightarrow \mathbb{CP}^{3}$ parametrized by $u\in\mathbb{CP}^{1}$. In formula
 \begin{equation*}
  l_{v_{0}}[1,u]=[1,u,f^{\top}(v_{0})-uf^{\bot c}(v_{0}),f^{\bot} (v_{0})+uf^{\top c}(v_{0})].
   \end{equation*}
 This is enough to prove the theorem since, from general facts about projective surfaces, we know that the number 
 of lines over a non-singular surface of degree greater or equal to 3 in $\mathbb{CP}^3$ is always finite.\end{proof}

\begin{remark}
The theory of lines or, in general, of rational curves over a surface is a very interesting and studied 
 field. In particular we point out that several further properties are stated about the nature
 of rational curves that can lie over a surface. Among the others we found~\cite[Theorem 1.1]{clemens} 
 and~\cite[Theorem 1]{xu},
 in which the authors state general formulas that imply that surfaces of degree greater
 or equal to 5 contain no lines.
For lower degrees we have that
the number of lines lying on a non-singular cubic surface is exactly $27$ (see e.g.~\cite{dolgacevcremona}),
while in the classical paper~\cite{segre} by Segre it is stated that the maximum number of lines lying on a
 non-singular quartic surface is $64$.

\end{remark}

\begin{remark}
 The case studied in~\cite{gensalsto} gave rise to a quartic ruled surface and so it is coherent with our last result.
\end{remark}

After the last result one can search for classes of singular varieties that can be reached by the twistor lift
of a slice regular function. Of course, since the argument of the proof is general, 
one can exclude from this classification all the surfaces which are not ruled by lines.
And so, we obtain the following theorems.

\begin{theorem}
Up to projective transformations, any quadric surface $\mathcal{Q}\subset\mathbb{CP}^3$ is such that
there exists a slice regular function $f$ such that its 
 twistor lift $\tilde{f}$ lies on $\mathcal{Q}$.
\end{theorem}
In the proof of this theorem, we will choose a particular union of two planes and a particular cone. Since the classification is projective, 
this is enough to complete all the possible cases. If one is interested in  singular quadric surfaces 
defined by different equations it may be possible to find no slice regular function whose lift realizes
the chosen equation.

\begin{proof}
The smooth case is solved thanks to Theorem~\ref{realquad} and by the fact that all non-singular 
quadric are projectively equivalent.
 Up to projective transformations there are only two classes of singular quadric surfaces: the union of two planes and cones.
 We will show that there is a cone and a union of two planes that can be described with coordinates in accordance
 with Equation~\eqref{twistorsplitting}.
 \begin{enumerate}
  \item Let $\mathcal{P}$ be the union of two planes defined by the following equation
  \begin{equation*}
   X_0^2-X_2^2=0.
  \end{equation*}
The slice regular function $f:\mathbb{H}\setminus\mathbb{R}\rightarrow\mathbb{H}$ defined by $f(\alpha+I\beta)=(\alpha+I\beta)(1-Ii)\frac{j}{2}$
lifts as $\tilde{f}[1,u,v,uv]=[1,u,1,v]$ and so lies it in $\mathcal{P}$.
\item Let $\mathcal{K}$ be the quadratic cone defined by the following equality
  \begin{equation*}
  X_1^2=X_2X_3
  \end{equation*}
Imposing then the usual equations we obtain that $G,H:\mathbb{C}_i\setminus\mathbb{R}\rightarrow\mathbb{C}_{i}$
can be chosen as
\begin{equation*}
 G(v)=\begin{cases}
                     0 & \mbox{ if }v\in\mathbb{C}_i^+\\
                     v & \mbox{ if }v\in\mathbb{C}_i^-.
                   \end{cases} ,\quad H(v)=\begin{cases}
                     0 & \mbox{ if }v\in\mathbb{C}_i^+\\
                     -\frac{1}{v} & \mbox{ if }v\in\mathbb{C}_i^-.
                   \end{cases}
\end{equation*}
 \end{enumerate}
As before, it is now a matter of computation, using the Representation Formula, to write the slice regular functions defined by the previous equation. 
\end{proof}

\subsection{Cubics}

We will treat now the case of cubics surfaces. Firstly we will consider \textit{non-normal}
cubics and then cones. An algebraic variety $X$ is said to be \textit{normal} if it is normal at every point, meaning that the local ring at any point is an integrally closed domain. 
If $X$ is a non-normal cubic surface, then its singular locus contains a 1 dimensional part (see~\cite[Chapter 9.2]{dolgacev}).

\begin{theorem}
 Let $\mathcal{C}$ be a non-normal cubic surface in $\mathbb{CP}^3$ that is not a cone. Then, up to projective isomorphisms, there exists a slice regular function $f$ such that its 
 twistor lift $\tilde{f}$ lies on $\mathcal{C}$.
\end{theorem}

\begin{proof}
 In~\cite[Theorem 9.2.1]{dolgacev}, the author says that, up to projective isomorphisms, the only non-normal cubic surfaces in $\mathbb{CP}^3$ that
 are not cones are the following two:
 \begin{enumerate}
  \item $X_0X_3^2+X_1^2X_2=0$,
  \item $X_0X_1X_3+X_2X_3^2+X_1^3=0$.
 \end{enumerate}
Setting the coordinates of the lift in Remark~\ref{usuallift} in the previous equations we obtain, respectively,
\begin{enumerate}
 \item $g(v)=-v^2$, $\hat{g}(v)=v$ and $h\equiv 0\equiv\hat{h}$
 \item $g(v)=-1/v$, $\hat{g}(v)=v$, $\hat{h}(v)=1/v^2$ and $h\equiv 0$
\end{enumerate}
and so, if $x=\alpha+I\beta\in\mathbb{H}\setminus\mathbb{R}$ and $v=\alpha+i\beta$, the two slice regular functions are, respectively,
\begin{enumerate}
 \item $f_1:\mathbb{H}\setminus\mathbb{R}\rightarrow\mathbb{H}$ defined by 
 \begin{equation}
  f_1(x)=-x^2\frac{(1-Ii)}{2}+x\frac{(1+Ii)}{2},
 \end{equation}
 \item $f_2:\mathbb{H}\setminus\mathbb{R}\rightarrow\mathbb{H}$ defined by 
 \begin{equation*}
  f_2(x)=-x^{-1}\frac{(1-Ii)}{2}+x\frac{(1+Ii)}{2}+x^{-2}\frac{(1+Ii)}{2}j
 \end{equation*}
\end{enumerate}
\end{proof}

The last case that we will treat is the case of cubic cones. The set of cubic cones can be defined by the equation
\begin{equation}\label{cubiccone}
 X_3^3-(c+1)X_3^2X_1+cX_3X_1^2-X_2^2X_1=0,
\end{equation}
where, if $c\in\mathbb{C}\setminus\{0,1\}$, the surface is a cone over a non-singular plane cubic curve, while, in the case in which $c=0,1$
the surface is a cone over a nodal or cuspidal plane, cubic curve, respectively.
\begin{theorem}
 Let $\mathcal{C}$ be a cubic cone. Then there exists a slice regular function $f$ defined on a
 dense subset of $\mathbb{H}\setminus\mathbb{R}$, such that,
 up to projective transformations, its 
 twistor lift $\tilde{f}$ lies on $\mathcal{C}$.
\end{theorem}
\begin{proof}
 As in the previous theorems we will prove this result by exhibiting the splitting of the function $f$.
If we impose Equation~\eqref{cubiccone} to hold for the splitting in Formula~\eqref{twistorsplitting}
we obtain that
 $g$ and $h$ must be identically zero while $\hat{g}$ and $\hat{h}$ must satisfy the following equation
 \begin{equation*}
  \hat{g}^3-(c+1)\hat{g}^2+c\hat{g}=\hat{h}^2.
 \end{equation*}
Solving then in $\hat{h}$ or in $\hat{g}$, one finds the desired splitting of the slice regular function that
give the thesis.
\end{proof}

Since, up to projective transformations, the only cubic surfaces that contain infinite lines are cones and the non-normal ones,
then, the projective classification is complete.

Of course, the functions seen in the previous proofs are not the only slice regular functions that solve the problem and give the thesis. One could ask for 
the ``best'' slice regular function such that its lift satisfies a certain algebraic equation, but this issue will not be treated in this paper and we 
 propose it for some future work.

\section{Rational curves on the Grassmannian}\label{ratcurvesec}

The aim of this section is to recover the \textit{twistor transform} defined in~\cite{gensalsto} for
slice regular functions that are not defined on the real line. Moreover at the end we will characterize
certain rational curves over the Grassmannian $\mathbb{G}r_{2}(\mathbb{C}^{4})$.

The non-singular quadric in Equation~\eqref{quadric} is biholomorphic to $\mathbb{CP}^1\times\mathbb{CP}^1$ and the rulings are parametrized by $u$ and $v$. 
A sphere $\alpha+\mathbb{S}\beta$ can be identified with the line,
\begin{equation*}
l_{v_0}:=\{[1,u,\alpha+i\beta,(\alpha+i\beta)u]\,|\,u\in\mathbb{C}\cup \{ \infty \} \}\subset\mathbb{CP}^{3},
\end{equation*}
defined by fixing $v_0=\alpha+i\beta$.
The line $l_{v_0}$ can also be seen as a point in the Grassmannian $\mathbb{G}r_{2}(\mathbb{C}^{4})$ or, equivalently, as 
a point in the Klein quadric in $\mathbb{P}(\bigwedge^{2}\mathbb{C}^{4})\simeq \mathbb{CP}^{5}$ via Pl\"ucker embedding.


As we saw in Section~\ref{twsec}, left multiplication by $j$ on $\mathbb{H}^{2}$ lifts in $\mathbb{C}^{4}$ as 
\begin{equation*}
[X_{0},X_{1},X_{2},X_{3}]\xrightarrow{j\cdot}[-\overline{X_{1}},\overline{X_{0}},-\overline{X_{3}},\overline{X_{2}}],
\end{equation*}
and the last induces a real structure $\sigma$ over $\mathbb{CP}^{5}$ as follows,
\begin{equation*}
\sigma: [\xi_{1},\dots,\xi_{6}]\mapsto[\bar\xi_{1},\bar\xi_{5},-\bar\xi_{4},-\bar\xi_{3},\bar\xi_{2},\bar\xi_{6}],
\end{equation*}
%
where $\{\xi_{1},\dots,\xi_{6}\}$ represent the basis $\{e^{01},e^{02},e^{03},e^{12},e^{13},e^{23}\}$ of $\bigwedge^{2}\mathbb{C}^{4}$ and, of course, $e^{ij}:=e^{i}\wedge e^{j}$.
In the above coordinates we can explicit the equation of the Klein quadric as follows,
\begin{equation}\label{klein}
\xi_{1}\xi_{6}-\xi_{2}\xi_{5}+\xi_{3}\xi_{4}=0.
\end{equation}
As explained previously in Section 3 (and in~\cite[Section 2]{shapiro}), 
\textit{a fixed point of $\sigma$ corresponds to a $j$-invariant line in $\mathbb{CP}^3$, i.e. a (twistor) fibre of $\pi$}.

\begin{example}
Consider the coordinates found in Theorem~\ref{thmquadric} as functions defined on $\mathbb{CP}^{1}\times \mathbb{CP}^{1}$. We want to find the twistor fibre mentioned in the 
previous result by imposing equation $\sigma(\mathcal{F}(v))=\mathcal{F}(v)$. 
\begin{enumerate}
 \item If $\lambda=\mu\neq 0$ and $\nu=\pi/2$ we get, $\mathcal{F}:v\mapsto [1,c(1-v^2)^{1/2},-v,v,\frac{1}{c}(1-v^2)^{1/2},1]$.
 Imposing $\sigma(\mathcal{F}(v))=\mathcal{F}(v)$, we obtain $v=\pm 1$ (i.e. \textit{two twistor lines} in correspondence of $x=\pm 1\in\mathbb{R}$).
 
 \item  If $\lambda=\mu= 0$ and $\nu\in(0,\pi/2)/2$ we get, 
 \begin{equation*}
 \mathcal{F}:v\mapsto [v^2-\frac{e^{2i\nu}+v^2}{||e^{i\nu}||^2},\frac{i}{||e^{i\nu}||}(e^{2i\nu}+v^2)^{1/2},-v,v,\frac{i}{||e^{i\nu}||}(e^{2i\nu}+v^2)^{1/2},1].
 \end{equation*}
 Imposing  $\sigma(\mathcal{F}(v))=\mathcal{F}(v)$, we obtain no solution or \textit{no twistor lines} (this because $\omega$ is a fixed non-real complex number).
 
 \item If $\mathcal{Q}$ is the zero set of the polynomial in Equation~\eqref{quadnondiag}, we get, $\mathcal{F}:v\mapsto [-(\frac{5}{4}v^2+4),2i+\frac{v}{2},-v,-v,2i-\frac{v}{2},1]$.
 Imposing  $\sigma(\mathcal{F}(v))=\mathcal{F}(v)$, we obtain $v=-4i$ (i.e. \textit{one twistor line} in correspondence of $x=-4i\in\mathbb{H}$).
 \end{enumerate}
\end{example}
At this point we can extend the definition given in~\cite{gensalsto} of twistor transform.
\begin{definition}[Twistor transform]
Let $D\subset\mathbb{C}^{+}$ be a domain and  $f:\Omega_{D}\rightarrow\mathbb{H}$ be a slice  function. We  define the  \textit{twistor transform} of $f$ as the following map:
\begin{equation*}
\begin{array}{rrcl}
\mathcal{F}: & D & \rightarrow &\mathbb{G}r_2(\mathbb{C}^{4})\\
& v & \mapsto & \tilde{f}(l_{v}).
\end{array}
\end{equation*}
\end{definition}

The following result extends~\cite[Theorem 5.7]{gensalsto}.

\begin{theorem}
Let $D$ be a domain in $\mathbb{C}^{+}$. If $f:\Omega_{D}\rightarrow \mathbb{H}$ is a slice  function, then its twistor transform $\mathcal{F}$ defines a curve over $D$. 
Moreover, every  curve 
$\gamma: D\rightarrow\mathbb{G}r(\mathbb{C}^{4})$, such that $\xi_{6}\circ \gamma$ is never zero, 
is the twistor transform of a slice  function $f:\Omega_{D}\rightarrow \mathbb{H}$.
The function $f$ is regular if and only if its twistor transform is a holomorphic curve.
\end{theorem}

\begin{proof}
Given a slice  function $f:\Omega_D\rightarrow \mathbb{H}$, its twistor lift is given, as in Formula~\eqref{lift}, by, 
$\tilde{f}[1,u,v,uv]=[1,u,p_{i}\circ F^{\top}(v)-u(p_{i}\circ F^{\bot c}(v)),p_{i}\circ F^{\bot} (v)+u(p_{i}\circ F^{\top c}(v))]$,
where $f^{\top}$ and $f^{\bot}$ are
defined as in the proof of Theorem~\ref{thmlift}. Fixing $v$, $\tilde{f}(l_v)$ is defined by the following
linear equations:
\begin{equation*}
 \begin{cases}
X_0(p_{i}\circ F^{\top})-X_1(p_{i}\circ F^{\bot c})-X_2=0\\
X_0(p_{i}\circ F^{\bot}) +X_1(p_{i}\circ F^{\top c})-X_3=0.
\end{cases}
\end{equation*}
The coefficients of the last two equations determine the following generating vectors
\begin{equation*}
 e_1=[p_{i}\circ F^{\top},-p_{i}\circ F^{\bot c},-1,0],\quad e_2=[p_{i}\circ F^{\bot},p_{i}\circ F^{\top c},0,-1].
\end{equation*}
Using Equation~\eqref{klein}, then, the twistor transform can be made explicit as follows
\begin{multline*}
\mathcal{F}(v)=[\xi_1,\dots,\xi_6]=
[(p_{i}\circ F^{\top})(v)(p_{i}\circ F^{\top c})(v)+(p_{i}\circ F^{\bot})(v)(p_{i}\circ F^{\bot c})(v),\\(p_{i}\circ F^{\bot})(v),-(p_{i}\circ F^{\top})(v),(p_{i}\circ F^{\top c})(v),(p_{i}\circ F^{\bot c})(v),1],
\end{multline*}
where $\{\xi_i\}=\{e_1^h\wedge e_2^k\}_{0\leq h<k\leq 3}$. But now that we have the explicit parameterization of $\mathcal{F}(v)$ it is clear that this is a
holomorphic curve if and only if $f$ is a slice regular function.\\
Vice versa, given a  curve $\gamma:D\rightarrow \mathbb{G}r_2(\mathbb{C}^4)$ such that $\xi_6\circ \gamma$ is never zero, we can assume $\xi_6\circ \gamma=1$
and recover the splittings of $f$ as follows,
\begin{equation*}
 (p_{i}\circ F^{\top})=-\xi_3\circ\gamma,\quad (p_{i}\circ F^{\bot})=\xi_2\circ\gamma,\quad (p_{i}\circ F^{\top c})=\xi_4\circ\gamma,\quad (p_{i}\circ F^{\bot c})=\xi_5\circ\gamma.
\end{equation*}
Thanks to the Representation Theorem we can now recover $f$ and Remark~\ref{charareg} give us regularity.
\end{proof}

From the proof, then, we have that the twistor transform $\mathcal{F}$ of a slice regular function $f$, can be represented in the following way,
\begin{multline*}
\mathcal{F}(v)=[(p_{i}\circ F^{\top})(v)(p_{i}\circ F^{\top c})(v)+(p_{i}\circ F^{\bot})(v)(p_{i}\circ F^{\bot c})(v),\\(p_{i}\circ F^{\bot})(v),-(p_{i}\circ F^{\top})(v),(p_{i}\circ F^{\top c})(v),(p_{i}\circ F^{\bot c})(v),1].
\end{multline*}

\begin{remark}
As for Theorem~\ref{thmlift}, in the last proof we could repeat the computations using the 
Splitting Lemma. The result would be the following,
\begin{equation*}
\mathcal{F}(v)=[g(v)\hat{g}(v)+\hat{h}(v)h(v),h(v),-g(v),\hat{g}(v),\hat{h}(v),1],
\end{equation*}
which coincides with the result in~\cite{gensalsto}.
\end{remark}

We will now present some examples.
\begin{example}
 
\begin{itemize}
\item Let $f_{1}:\mathbb{H}\setminus \mathbb{R}\rightarrow \mathbb{H}$ be the following slice regular function: 
$f(\alpha+ I\beta)=1-Ii$. This function is equal to $2$ over $\mathbb{C}_{i}$ and to $0$ over $\mathbb{C}_{-i}$. Its twistor transform 
$\mathcal{F}_{1}:\mathbb{C}^{+}\rightarrow\mathbb{G}r(\mathbb{C}^{4})$ is the constant function 
$v\mapsto [0,0,-2,0,0,1]$.
\item Let $f_{2}:\mathbb{H}\setminus \mathbb{R}\rightarrow \mathbb{H}$ be the following slice regular function: 
$f(\alpha+ I\beta)=1+Ii$. This function is equal to $0$ over $\mathbb{C}_{i}$ and to $2$ over $\mathbb{C}_{-i}$. 
Its twistor transform $\mathcal{F}_{2}:\mathbb{C}^{+}\rightarrow\mathbb{G}r(\mathbb{C}^{4})$ is the constant function 
$v\mapsto [0,0,0,2,0,1]$.
\item Let $f_{3}:\mathbb{H}\setminus \mathbb{R}\rightarrow \mathbb{H}$ be the following slice regular function: 
$f(\alpha+ I\beta)=(\alpha+I\beta)(1-Ii)/2$. This function is equal to $(\alpha+I\beta)$ over $\mathbb{C}_{i}$ and to $0$ over $\mathbb{C}_{-i}$. 
Its twistor transform $\mathcal{F}_{3}:\mathbb{C}^{+}\rightarrow\mathbb{G}r(\mathbb{C}^{4})$ is the function 
$v\mapsto [0,0,-v,0,0,1]$.
\item Let $f_{4}:\mathbb{H}\setminus \mathbb{R}\rightarrow \mathbb{H}$ be the following slice regular function: 
$f(\alpha+ I\beta)=(\alpha+I\beta)(1+Ii)/2$. This function is equal to $0$ over $\mathbb{C}_{i}$ and to $(\alpha+I\beta)$ over $\mathbb{C}_{-i}$. 
Its twistor transform $\mathcal{F}_{4}:\mathbb{C}^{+}\rightarrow\mathbb{G}r(\mathbb{C}^{4})$ is the function 
$v\mapsto [0,0,0,v,0,1]$.
\end{itemize}

\end{example}

As said at the beginning of this section we want to characterize a certain class of linear holomorphic functions 
$\gamma: D\rightarrow\mathbb{G}r(\mathbb{C}^{4})$ in terms 
of slice regular functions. We will restrict to the case in which $\xi_{6}\circ \gamma$ is never zero, 
The theorem we are going to prove is the following.

\begin{theorem}
Let $\gamma: \mathbb{C}^+\rightarrow\mathbb{G}r(\mathbb{C}^{4})$ be a holomorphic curve such that $\xi_{6}\circ \gamma$ is never zero. Then $\gamma$ is affine   if and only if 
there exist  $A,B\in\mathbb{C}$, with $A/B\in\mathbb{C}^{+}\cup\mathbb{R}$ such that 
$\gamma$ is the twistor 
transform of a slice regular function $f$ and $(A+xB)\cdot f$ is a slice affine function that satisfies
\begin{equation}\label{herm}
 h_i(Af_i-Bg_i,\bar{A}f_{-i}-\bar{B}g_{-i})=0,
\end{equation}
 where $f_{\pm i}$ are the values of the slice derivative of $(A+xB)\cdot f$ in $\mathbb{C}_{\pm i}$,
 $g_{\pm i}$ are the values of the slice constant function $(A+xB)\cdot f-x[(1-Ii)f_i+(1+Ii)f_{-i}]$ in $\mathbb{C}_{\pm i}$
and $h_i$ denotes the hermitian product in $\mathbb{C}_i\oplus\mathbb{C}_{i}^\bot\simeq\mathbb{H}$.
\end{theorem}

\begin{proof}
 A linear map $\gamma: \mathbb{C}^+\rightarrow\mathbb{G}r(\mathbb{C}^{4})$ is a map of the form,
 \begin{equation*}
  \gamma(v)=[c_{11}+c_{12}v,c_{21}+c_{22}v,c_{31}+c_{32}v,c_{41}+c_{42}v,c_{51}+c_{52}v,c_{61}+c_{62}v],
 \end{equation*}
considering the Grassmannian $\mathbb{G}r_2(\mathbb{C}^4)$ as the Klein quadric given in Formula~\eqref{klein} in $\mathbb{CP}^5$. The condition $\xi_6\circ \gamma\neq 0$ for all $v\in\mathbb{C}^{+}$ can be interpreted,
of course, as $c_{61}/c_{62}\in\mathbb{C}^{+}\cup\mathbb{R}$. Dividing everything by 
$c_{61}+c_{62}v$, we obtain
 \begin{equation*}
  \gamma(v)=\left[\frac{c_{11}+c_{12}v}{c_{61}+c_{62}v},\frac{c_{21}+c_{22}v}{c_{61}+c_{62}v},\frac{c_{31}+c_{32}v}{c_{61}+c_{62}v},\frac{c_{41}+c_{42}v}{c_{61}+c_{62}v},\frac{c_{51}+c_{52}v}{c_{61}+c_{62}v},1\right],
 \end{equation*}
and so, now $\xi_6\circ \gamma=1$. Substituting then the components of $\gamma$ in Equation~\eqref{klein}, one obtain the following system of equations:
\begin{equation}\label{sysrat}
\left\{ \begin{array}{l}
  c_{11}c_{61}- c_{21}c_{51}+c_{31}c_{41}=0\\
  c_{11}c_{62}+c_{12}c_{61}- (c_{21}c_{52}+c_{22}c_{51})+(c_{31}c_{42}+c_{32}c_{41})=0\\
  c_{12}c_{62}+c_{32}c_{42}+ c_{22}c_{52}=0
 \end{array}\right. .
\end{equation}
Moreover, since $\gamma$ is a holomorphic function, then it will be the twistor transform of some slice regular function $f$ such that
\begin{equation*}
 \begin{array}{rcl}
  f_{\mathbb{C}_i^+}(\alpha+i\beta) & = & 
  -\displaystyle\frac{c_{31}+c_{32}(\alpha+i\beta)}{c_{61}+c_{62}(\alpha+i\beta)}+
  \displaystyle\frac{c_{21}+c_{22}(\alpha+i\beta)}{c_{61}+c_{62}(\alpha+i\beta)}j\\
  & & \\
  f_{\mathbb{C}_{-i}^+}(\alpha-i\beta) & = & 
  \displaystyle\overline{\frac{c_{41}+c_{42}(\alpha+i\beta)}{c_{61}+c_{62}(\alpha+i\beta)}}+
  \displaystyle\overline{\frac{c_{51}+c_{52}(\alpha+i\beta)}{c_{61}+c_{62}(\alpha+i\beta)}}j.
 \end{array}
\end{equation*}
Thanks to the Representation Formula one obtains that, for each $\alpha+ I\beta\in\mathbb{H}\setminus\mathbb{R}$,
\begin{multline*}
 2f(\alpha+I\beta) = [(1-Ii)f(\alpha+i\beta)+(1+Ii)f(\alpha-Ii)]=\\
  =  (c_{61}+(\alpha+I\beta)c_{62})^{-\cdot}\cdot [(\alpha+I\beta)(1-Ii)(-c_{32}+c_{22}j)+(1-Ii)(-c_{31}+c_{21}j)]+\\
 +(c_{61}+(\alpha+I\beta)c_{62})^{-\cdot}\cdot[(\alpha+I\beta)(1+Ii)(\bar c_{42}+\bar c_{52}j)+(1+Ii)(\bar c_{41}+\bar c_{51}j)],
\end{multline*}
but then, $(c_{61}+(\alpha+I\beta)c_{62})\cdot f$ is a slice affine function.
If now, one between $c_{61}$ or $c_{62}$ is equal to zero this correspond, respectively, to $A$ or $B$ equal to zero and
so Equation~\eqref{herm} holds true.
If both $c_{61}$ and $c_{62}$ are non-zero, observe that, the first and the third equations in Formula~\eqref{sysrat} 
can be written , respectively, as $h_i(g_i,g_{-i})=c_{11}A$ and $h_i(f_i,f_{-i})=c_{12}B$.
Substituting these in the second equation of the system and since
$ (c_{21}c_{52}+c_{22}c_{51})-(c_{31}c_{42}+c_{32}c_{41})=h_i(g_i,f_{-i})+h_i(f_i,g_{-i})$,
we get
\begin{equation*}
 h_i(g_i,g_{-i})\frac{B}{A}+h_i(f_i,f_{-i})\frac{A}{B}=h_i(g_i,f_{-i})+h_i(f_i,g_{-i}),
\end{equation*}
and so Equation~\eqref{herm} holds true.
The vice versa is trivial.
\end{proof}
\begin{example}
Simple examples of slice regular functions that satisfy the condition in Equation~\eqref{herm},
are all the functions of the following type:
\begin{equation*}
 \begin{array}{rclc}
  f: & \mathbb{H}\setminus\mathbb{R} & \rightarrow & \mathbb{H}\\
     & \alpha +I\beta & \mapsto & (Cx+D)^{-\cdot}\cdot(Ax+B)(1-Ii)/2,
 \end{array}
\end{equation*}
with $\left(\begin{array}{cc}
A & B \\
C & D
\end{array}\right)\in SL(2,\mathbb{R})$. In the next section we will study one particular function in this set and then we will add some remarks to the whole family.
\end{example}

\begin{remark}
 The set of slice affine functions that satisfy Formula~\eqref{herm} does not contain non constant slice functions that extend to the real line. In fact, as shown in Remark~\ref{extend}, a 
 slice affine function extends to $\mathbb{R}$ if the coefficients of first order are equal, i.e.: $f_+=f_-$, meaning that $ h_i(f_i,f_{-i})\neq 0$.
\end{remark}

\section{A first non trivial example}

In this section we will study the following slice regular function
\begin{equation}\label{function}
 \begin{array}{rclc}
  f: & \mathbb{H}\setminus\mathbb{R} & \rightarrow & \mathbb{H}\\
     & \alpha +I\beta & \mapsto & (\alpha+I\beta)(1-Ii)/2
 \end{array}
\end{equation}
as a tool to generate  OCS's over its image.
We will write also, for brevity, $f(x)=x(1-Ii)/2$, where $x=\alpha+I\beta\in\mathbb{H}\setminus\mathbb{R}$. As
was shown in~\cite{altavilla}, this function is constant and equal to $0$ if restricted to $\mathbb{C}_{-i}^+$ and equal to the identity if restricted
to $\mathbb{C}_i^+$. In the same paper it was shown either theoretically and by explicit computations that its restriction to $\mathbb{H}\setminus\mathbb{C}_{-i}^+$ is an
open function.
In~\cite{altavilladiff} it was proved that, if restricted to $\mathbb{H}\setminus\mathbb{C}_{-i}^+$, the function $f$ is injective. 
For these reasons this function fits very well in the twistorial construction studied here. 
Moreover, this construction has a symbiotic aspect with respect to the function $f$. In fact, with the help of
the twistor lift stated in Theorem~\ref{thmlift} it is possible to understand constructively the image of $f$.
The next theorem precises this fact.
\begin{theorem}\label{image}
 If $q=q_0+q_1i+q_2j+q_3k$, then the function defined in Equation~\eqref{function} is such that 
 $f(\mathbb{H}\setminus\mathbb{C}_{-i}^+)=\{q\in\mathbb{H}\,\mid\,q_1>0\}$. Moreover 
 \begin{equation*}
  \bigcup_{I\in\mathbb{S}}f\mid_{\mathbb{C}_I^+}(\mathbb{R})=\{q\in\mathbb{H}\,\mid\,q_1=0\},
 \end{equation*}
 where $f\mid_{\mathbb{C}_I^+}(\mathbb{R})$ means the unique extension to $\mathbb{R}$ of the function restricted to $\mathbb{C}_I^+$.
\end{theorem}

\begin{proof}
 To prove the theorem we will use the twistor lift in Formula~\eqref{lift}. In fact, thanks to Theorem~\ref{thmlift}, it is possible to compute the image of a slice regular function by looking
 at the image of the projection to $\mathbb{H}$ of its twistor lift. Since, as already said, the function $f$ is equal to the identity if restricted to
 $\mathbb{C}_i^+$ and to zero over the opposite semislice $\mathbb{C}_{-i}^+$, then its twistor lift is defined as follows:
 \begin{equation}\label{twistorf}
 \begin{array}{rclc}
  F: & \mathcal{Q}^+\cap\pi^{-1}(\mathbb{H}\setminus\mathbb{C}_{-i}^+) & \rightarrow & \mathbb{CP}^3\\
     & [1,u,v,uv] & \mapsto & [1,u,v,0],
 \end{array}
\end{equation}
where, if $\alpha+I\beta\in\mathbb{H}\setminus\mathbb{C}_i^+$ and $I=ai+bj+ck$, then $u=-i\frac{b+ic}{a+1}$ and $v=\alpha+i\beta$, with $(a,b,c)\neq (-1,0,0)$ and $\beta>0$.
At the end what we want to compute is the image of the function $(1+uj)^{-1}v$ and so these are the computations:
\begin{equation*}
\begin{array}{rcl}
 (1+uj)^{-1}v & = & \left(1-\displaystyle\frac{b+ic}{a+1}k\right)(\alpha+i\beta)\\
 & = & \displaystyle\frac{(a+1)^2}{(a+1)^2+(b^2+c^2)}\left(1+\displaystyle\frac{bk-cj}{a+1}\right)(\alpha+i\beta)\\
 & = & \displaystyle\frac{1}{2}[(a+1)(\alpha+i\beta)+(\beta b-\alpha c)j+(\alpha b+\beta c)k].
\end{array}
\end{equation*}
So, the image  of a quaternion $x=\alpha+(ai+bj+ck)\beta$ via $f$, with $ai+bc+ck\in\mathbb{S}\setminus\{-i\}$ and $\beta>0$ is the quaternion
\begin{equation*}
 2f(x)=\alpha(a+1)+\beta(a+1)i+(\beta b-\alpha c)j+(\alpha b+\beta c)k.
\end{equation*}
Take now a generic quaternion $q=q_0+q_1i+q_2j+q_3k$. This will be reached by $f$ if and only if $q_1>0$. In fact the system
\begin{equation*}
 \begin{cases}
  \alpha(a+1)=q_0\\
  \beta(a+1)=q_1\\
  \beta b-\alpha c=q_2\\
  \alpha b+\beta c=q_3,
 \end{cases}
\end{equation*}
can be solved in the following way: the first two equations give $\alpha=q_0/(a+1)$ and $\beta=q_1/(a+1)$ and since $(a+1)\in(0,2]$, then $q_1>0$. If we set $B=b/(a+1)$ and $C=c/(a+1)$,
the last two equations can be written as 
\begin{equation*}
 \begin{cases}
  q_1B-q_0C=q_2\\
  q_0C+q_1B=q_3.
 \end{cases}
\end{equation*}
The last is a linear system such that the  two equations are linearly independent, so the solutions is,
\begin{equation*}
 B=\frac{q_1q_2+q_0q_3}{q_0^2+q_1^2},\quad C=\frac{q_1q_3-q_0q_2}{q_0^2+q_1^2}.
\end{equation*}
Now we remember that $a^2+b^2+c^2=1$ and so $B^2+C^2=\frac{1-a}{1+a}$ that entails $a=\frac{1-B^2-C^2}{1+B^2+C^2}$ which is always an admissible solution since it is always different
from $-1$.

For the second part of the theorem, fix $I=ai+bj+ck\in\mathbb{S}\setminus\{-i\}$ and look for the following limit,
\begin{equation*}
  \lim_{\underset{\alpha+I\beta\in\mathbb{C}_I^+}{\beta\to 0}} f(\alpha+I\beta).
\end{equation*}
After restricting the function to $\mathbb{C}_I^+$ it is possible to extend it to $\mathbb{R}$ and also to look at the image via the twistor lift. Since $f$ is continuous we obtain
that, up to a factor 2, the previous limit is equal to
\begin{equation*}
\alpha(a+1)-\alpha cj+\alpha bk=\alpha (a+1,0,-c,b),
\end{equation*}
which is a straight line belonging to the set $\{q\in\mathbb{H}\,\mid\,q_1=0\}$ passing through the vector $(a+1,0,-c,b)$. Taking the union, for $(a,b,c)$ that runs over 
$\mathbb{S}\setminus\{-i\}$, it is clear that this will span the whole hyperplane $\{q_1=0\}$.
\end{proof}

The twistor lift of $f$ lies in the hypersurface $\mathcal{H}:=\{X_{3}=0\}\subset\mathbb{CP}^3$. In this case the general theory (see Section 3 of the present paper and~\cite[Section 3]{salamonviac}) says that 
$\mathcal{H}$ induces an OCS conformally equivalent to a constant one, defined over the image of $f$. This is actually true and we will show that there is a specific 
conformal function from $\{q_1>0\}\subset \mathbb{H}$ to $\{q_1<0\}$ that sends $\mathbb{J}^f$ to $i$.
The theorem is the following one.
\begin{theorem}
 The complex metric manifold $(\{q_1>0\},g_{Eucl},\mathbb{J}^f)$ is conformally equivalent to $(\{q_1<0\},g_{Eucl},\mathbb{J}_i)$, where, by $\mathbb{J}_i$ we mean 
 the left multiplication by $i$. The conformality is determined by the function $g:\{q_1>0\}\rightarrow \{q_1<0\}$ defined by $g(q)=q^{-1}$.
\end{theorem}

\begin{proof}
 The function $g$ is of course a conformal map for the Euclidean metric. So, the only thing to prove is that the push-forward of $\mathbb{J}^f$ via $g$ is exactly
 $\mathbb{J}_i$, meaning that, the following equality holds true
 \begin{equation*}
  dg\circ\mathbb{J}^f=\mathbb{J}_i\circ dg.
 \end{equation*}
 We compute now the $4\times 4$ matrices representing the two complex structures $\mathbb{J}^f$ and $\mathbb{J}_i$.
%
%
%
We have that, if $v=(v_0,v_1,v_2,v_3)$ is a tangent vector in $p=f(\alpha+I\beta)$, then, $\mathbb{J}_i(p)v=(-v_1,v_0,-v_3,v_2)$, while
$\mathbb{J}^f(p)v=(-av_1-bv_2-cv_3,av_0-cv_2+bv_3,bv_0+cv_1-av_3,cv_0-bv_1+av_2)$, where $ai+bj+ck=I$.
Therefore we have that
\begin{equation*}
 \mathbb{J}_i=\left(\begin{array}{cccc}
                     0 & -1 & 0 & 0\\
                     1 & 0 & 0 & 0\\
                     0 & 0 &  0 & -1\\
                     0 & 0 & 1 & 0
                    \end{array}
\right),\quad
\mathbb{J}^f(p)=\left(\begin{array}{cccc}
                     0 & -a & -b & -c\\
                     a & 0 & -c & b\\
                     b & c &  0 & -a\\
                     c & -b & a & 0
                    \end{array}
\right),
\end{equation*}
where $p=p_0+p_1i+p_2j+p_3k$ (see also Remark~\ref{matrix}) and, working on the computations in the proof of Theorem~\ref{image},
\begin{equation*}
 a=\frac{p_0^2+p_1^2-p_2^2-p_3^2}{\mid p\mid^2},\quad b=2\frac{p_0p_3+p_1p_2}{\mid p\mid^2},\quad c=2\frac{p_1p_3-p_0p_2}{\mid p\mid^2}.
\end{equation*}
Now, writing $g$ as $g(q_0+q_1i+q_2j+q_3k)=(q_0,-q_1,-q_2,-q_3)/\mid q\mid^2$, one has that
\begin{equation*}
 dg(q)=\left(\begin{array}{cccc}
                     \mid q\mid^2-2q_0^2 & -2q_0q_1 & -2q_0q_2 & -2q_0q_3\\
                     2q_1q_0 & -\mid q\mid^2+2q_1^2 & 2q_1q_2 & 2q_1q_3\\
                     2q_2q_0 & 2q_2q_1 &  -\mid q\mid^2+2q_2^2 & 2q_2q_3\\
                     2q_3q_0 & 2q_3q_1 & 2q_3q_2 & -\mid q\mid^2+q_3^2
                    \end{array}\right)/\mid q\mid^4
\end{equation*}
and that,
\begin{equation*}
(\mathbb{J}_i\circ dg)(q)=\left(\begin{array}{cccc}
		      -2q_1q_0 & \mid q\mid^2-2q_1^2 & -2q_1q_2 & -2q_1q_3\\
                     \mid q\mid^2-2q_0^2 & -2q_0q_1 & -2q_0q_2 & -2q_0q_3\\
                     -2q_3q_0 & -2q_3q_1 & -2q_3q_2 & \mid q\mid^2-2q_3^2\\
                     2q_2q_0 & 2q_2q_1 &  -\mid q\mid^2+2q_2^2 & 2q_2q_3\\
                    \end{array}\right)/\mid q\mid^4.
\end{equation*}
We leave to the reader the (long but easy) computation of $(dg\circ\mathbb{J}^f)(q)$ and to 
check that $dg\circ\mathbb{J}^f=\mathbb{J}_i\circ dg$.
%
%
%
\end{proof}

The previous Theorem implies, in particular, the existence of a biholomorphism between 
the two complex manifolds $(\mathbb{H}\setminus\mathbb{C}_{-i}^+,\mathbb{J})$ and $(\mathbb{C}^2,i)$.

\begin{remark}
 The function $g(q)=q^{-1}$ in the previous theorem, was found using the following idea. The constant OCS $\mathbb{J}_i$ is described by the hyperplane 
 $\{X_1=0\}\subset\mathbb{CP}^3$ (see~\cite[Remark 2.3]{salamonviac}) and so, starting from our lift $[1,u,v,0]$ after changing the first two coordinates with the second two and
 dividing everything by $v(\neq 0)$, we obtain $[1,0,v^{-1},v^{-1}u]$ that projects to $[1,v^{-1}(1+uj)]$, but now $v^{-1}(1+uj)=((1+uj)^{-1}v)^{-1}=(f(q))^{-1}$.
\end{remark}

\begin{remark}
The last theorem and construction can be obtained using the following function as well: 
$f:\mathbb{H}\setminus\mathbb{R}\rightarrow\mathbb{H}$, defined as
\begin{equation*}
f(\alpha+I\beta)=(Cx+D)^{-\cdot}\cdot(Ax+B)\frac{(1-Ii)}{2},
\end{equation*}
with  $\left(\begin{array}{cc}
A & B \\
C & D
\end{array}\right)\in SL(2,\mathbb{R})$, $x=\alpha+I\beta$ and $z=\alpha+i\beta$. In fact, if we remove from the domain of this function the semislice $\mathbb{C}_{-i}^{+}$ over which is 
equal to zero, $f$ is open and injective and its image is equal again to $\{q\in\mathbb{H}\,\mid\,q_1>0\}$. With easy computations one obtains that 
\begin{equation*}
 f(\alpha+I\beta)=\begin{cases}
 \displaystyle\frac{(a+1)}{2\|Cz+D\|^{2}}[CA\|z\|^{2}+DB+(BC+AD)\alpha]=q_0\\
  \\
\displaystyle\frac{(a+1)}{2\|Cz+D\|^{2}}\beta=q_1\\
 \\
 \displaystyle\frac{(b\beta-c[CA\|z\|^{2}+DB+(BC+AD)\alpha])}{2\|Cz+D\|^{2}} =q_2\\
 \\
\displaystyle\frac{c\beta+b[CA\|z\|^{2}+DB+(BC+AD)\alpha]}{2\|Cz+D\|^{2}} =q_3,
 \end{cases},\quad z=\alpha+i\beta,
\end{equation*}
and, with the same argument in the proof of Theorem~\ref{image}, one obtains that $q_{1}>0$ and, for any values of $q_{0},q_{1}$, each $q_{2}$ and $q_{3}$ can be reached.
Now, on the remaining first two components the function is exactly equal to 
\begin{equation*}
\frac{A(\alpha+i\beta)+B}{C(\alpha+i\beta)+D}=\frac{q_{0}+iq_{1}}{(a+1)}.
\end{equation*}
Now, since $A,B,C,D$ are taken such that the matrix they describe is in $SL(2,\mathbb{R})$,
and since the function on the left describes an automorphism of the upper half complex space, it turns out that each  $q_{0}$ and $q_{1}>0$ can be reached.
The twistor lift of this function is 
 \begin{equation*}
 \begin{array}{rclc}
  \tilde{f}: & \mathcal{Q}^+\cap\pi^{-1}(\mathbb{H}\setminus\mathbb{C}_{-i}^+) & \rightarrow & \mathbb{CP}^3\\
     & [1,u,v,uv] & \mapsto & [1,u,\frac{Av+B}{Cv+D},0].
 \end{array}
\end{equation*}
\end{remark}

In the next remark we will show an idea that we have not explored completely but that might be a starting point
for some future considerations.

\begin{remark}
The twistor lift in Equation~\eqref{twistorf}, extends to a holomorphic mapping $\tilde{f}:\mathcal{Q}\rightarrow \mathbb{CP}^3$ by allowing $v$ to take values in $\mathbb{C}$ rather than 
just in $\mathbb{C}^+$. However, even if $\pi\circ \tilde{f}=f\circ \pi$ on  $\mathcal{Q}^+\cap\pi^{-1}(\mathbb{H}\setminus\mathbb{C}_{-i}^+)$,  
$$
\begindc{\commdiag}[50]
\obj(0,10)[aa]{$\mathcal{Q}^+$}
\obj(20,10)[bb]{$\{X_3=0\}$}
\obj(0,0)[cc]{$\mathbb{H}\setminus\mathbb{R}$}
\obj(20,0)[dd]{$\{q_1>0\}$}
\mor{aa}{bb}{$\tilde{f}$}
\mor{aa}{cc}{$\pi$}[-1,0]
\mor{bb}{dd}{$\pi$}
\mor{cc}{dd}{$f$}
\enddc
$$
this does not imply that the graph will commute once $\tilde{f}$ is extended. In fact we will have the following diagram, 
$$
\begindc{\commdiag}[50]
\obj(0,10)[aa]{$\mathcal{Q}$}
\obj(20,10)[bb]{$\{X_3=0\}$}
\obj(0,0)[cc]{$?$}
\obj(20,0)[dd]{$?$}
\mor{aa}{bb}{$1:1$}
\mor{aa}{cc}{$2:1$}[-1,0]
\mor{bb}{dd}{$1:1$}
\mor{cc}{dd}{$^*1:2^*$}[1,\dotline]
\enddc
$$
where, the numbers upon the arrows are intended as \textit{generically} and we do not know \textit{a priori} 
what to put in the two vertices below and what is the meaning of the arrow that connects them. 
Also this arrow must represent something which 
behaves like $1:2$. This of course cannot be possible and suggest the possibility of approaching  the issue using \textit{multi-valued} functions. Anyway this example seems enough easy 
to be studied directly. So, first of all, we need to construct the ``ghost function'' that realizes the second part of that $1:2$ cited before.
Therefore, when we extend $\tilde{f}$ to the whole $\mathcal{Q}$ we need the function that realizes the lifting $\tilde{f}[1,u,v,uv]=[1,u,v,0]$, for $v\in\mathbb{C}^-$.

In a certain sense, once you decompose the function in its four real components, the variable 
$\beta$ and $I$ are able to move independently.
So, depending on the interpretation one gives to the point $x(=\alpha+I\beta=\alpha+(-I)(-\beta))$,
the representation of the function $f$ in its vectorial form, returns two different values.
\end{remark}

\section{Conclusion and future works}
In this paper, after a brief review of slice regular functions and twistor space of $\mathbb{S}^{4}$,
we have shown that the theory introduced in~\cite{gensalsto} linking these two fields,
can be extended to all slice regular functions.
Moreover, the techniques used to extend the theory of slice regular functions to domains with
empty intersection with the real line were used to show a number of new results, such as
the second part of Theorem~\ref{thmlift}.
In this framework we have proved that this theory is effective in giving coordinates for 
the quadric surfaces in the conformal classification of non-singular 
quadrics in Theorem~\ref{classification} and we gave a projective classification of the remaining
quadrics and cubics that can be reached by the lift of a slice regular function.
Finally we have used all this material to show the effectiveness of these instruments 
in the task of finding an explicit biholomorphism between two particular complex 4-manifolds.

We hope to obtain further results in this direction. Some open problems
that we would like to explore in the future (some of them are, actually, work in progress) regard 
the conformal classification of remaining (singular) quadrics and cubics. Strongly linked to this
problem, we plan to solve the ambiguous ``multifunction'' issue contained in the last remark.
Furthermore we would like to further study the geometry of lines expressed by the twistor 
transform. In particular it would be interesting to classify other classes of rational curves
over the Grassmannian.

\section*{Acknowledgements}
The content of this paper was mainly developed during my Ph.D. studies at the Department
of Mathematics of the University of Trento. For this reason I would like to thank my supervisor 
Prof. A. Perotti and all the people (both academic and non-academic), who helped me in this work.
A special mention goes to Prof. G. Gentili, Dr. C. Stoppato, Prof. I. Sabadini, Prof. F. Colombo, Prof. F. Vlacci, Dr. G. Sarfatti, Prof. C. Bisi
and to Prof. E. Ballico for the useful discussions about classical algebraic geometry.
Furthermore, I thank King's College of London where part of the present work was done
and in particular Prof. S. Salamon for his useful suggestions. 
Moreover, this version of the manuscript was highly improved also thanks to the accuracy and the indications of the anonymous referee
and of Prof. C. de Fabritiis whose honest comments will always be welcome.

I was also partially supported by the Project  FIRB ``Geometria Differenziale e Teoria Geometrica delle Funzioni'', and by GNSAGA of INdAM.


\section*{References}

\begin{thebibliography}{00}


\bibitem{altavillaPhD}
Altavilla A. Quaternionic slice regular functions on domains without real points.
Ph.D. Thesis, supervisor A. Perotti, University of Trento, 2014,
http://eprints-phd.biblio.unitn.it/1089



\bibitem{altavilla}
Altavilla A. 
Some properties for quaternionic slice-regular functions on domains without real points. 
Complex Var. Elliptic Equ. 60, No. 1, 59-77 (2015).

\bibitem{altavilladiff}
Altavilla A. 
On the real differential of a slice regular function, to appear in Adv. in Geom. Preprint available at arXiv:1402.3993 [math.CV]


\bibitem{armpovsal}
Armstrong J, Povero M, Salamon S.
Twistor lines on cubic surfaces. Rend. Semin. Mat. Univ. Politec. Torino 71 (2013), no. 3-4, 317?338. 

\bibitem{armsal}
Armstrong J, Salamon S.
Twistor topology of the Fermat cubic. 
SIGMA Symmetry Integrability Geom. Methods Appl. 10 (2014), Paper 061, 12 pp. 





\bibitem{clemens}
Clemens H.
Curves on generic hypersurfaces.
Ann. Sci. \`Ecole Norm. Sup. (4) 19 (1986), no. 4, 629--636. 

\bibitem{colgensabstru}
Colombo F, Gentili G, Sabadini I, Struppa DC.
Extension results for slice regular functions of a quaternionic variable. 
Adv. Math. 222 (2009), no. 5, 1793--1808. 

\bibitem{cullen}
Cullen CG. An integral theorem for analytic intrinsic functions on quaternions. Duke Math J.. 1965; 32: pp. 139--148 


\bibitem{bohrsarfatti}
Della Rocchetta C, Gentili G, Sarfatti G. The Bohr Theorem for slice regular functions. Mathematische Nachrichten, vol. 285 (17-18), p. 2093-2105 (2012) doi: 10.1002/mana.201100232


\bibitem{blochlandau}
Della Rocchetta C, Gentili G, Sarfatti G.
A Bloch-Landau theorem for slice regular functions.
Advances in Hypercomplex Analysis (I. Sabadini, D.C. Struppa, G. Gentili, M. Shapiro, and F. Sommen, editors), Springer INdAM Series, Springer, Milan, 2013

\bibitem{dolgacev}
Dolgachev I. 
Classical algebraic geometry.
A modern view. Cambridge University Press, Cambridge, 2012. xii+639 pp. ISBN: 978-1-107-01765-8 

\bibitem{dolgacevcremona}
Dolgachev I.
Luigi Cremona and cubic surfaces. (Italian summary) Luigi Cremona (1830--1903) (Italian), 55--70,
Incontr. Studio, 36, Istituto Lombardo di Scienze e Lettere, Milan, 2005. 

\bibitem{landautoeplitz}
Gentili G, Sarfatti G.
Landau-Toeplitz theorems for slice regular functions over quaternions.
Pacific Journal of Mathematics, vol. 265, no. 2, p. 381-404 (2013) doi: 10.2140/pjm.2013.265.381


\bibitem{gensalsto}
Gentili G, Salamon S, Stoppato C. Twistor transforms of quaternionic functions and orthogonal complex structures.  
J. Eur. Math. Soc. (JEMS) 16, No. 11, 2323-2353 (2014).

\bibitem{genstoseries}
Gentili G, Stoppato C.
Power series and analyticity over the quaternions. 
Math. Ann. 352 (2012), no. 1, 113--131. 

\bibitem{gentilistoppato}
Gentili G, Stoppato C. The open mapping theorem for regular quaternionic functions. Ann. Sc. Norm. Super. Pisa Cl. Sci.. 2009; 5: pp. 805--815.

\bibitem{gentilistoppato2}
Gentili G, Stoppato C. Zeros of regular functions and polynomials of a quaternionic variable. Michigan Math. J. 56 (2008), no. 3, 655--667. 

\bibitem{genstostru}
Gentili G, Stoppato C, Struppa DC. Regular Functions of a Quaternionic Variable. Springer London, Limited.
Springer Monographs in Mathematics 2013.


\bibitem{gentilistruppa}
Gentili G, Struppa DC. A new theory of regular functions of a quaternionic variable. Adv. Math.. 2007; 216: pp. 279--301.





\bibitem{ghilmorper}
Ghiloni R, Moretti V, Perotti A.
Continuous slice functional calculus in quaternionic Hilbert spaces. 
Rev. Math. Phys. 25 (2013), no. 4, 1350006, 83 pp. 

\bibitem{ghiloniperotti}
Ghiloni R, Perotti A. 
Slice regular functions on real alternative algebras. Adv. Math.. 2011; 226: pp. 1662--1691.


\bibitem{ghiloniperotti2}
Ghiloni R, Perotti A. 
Global differential equations for slice regular functions. Math. Nach., 287, No. 5-6, Pages 561--573, 2014, DOI: 10.1002/mana.201200318

\bibitem{ghiloniperotti3}
Ghiloni R, Perotti A. 
Power and spherical series over real alternative *-algebras. 
Indiana University Mathematics Journal, Vol. 63 (2014), No. 2, Pages 495--532.



\bibitem{griffiths}
Griffiths P, Harris J.
Principles of algebraic geometry.
Pure and Applied Mathematics. Wiley-Interscience, New York, 1978. xii+813 pp. 

\bibitem{hartshorne}
Hartshorne R.
Algebraic geometry.
Graduate Texts in Mathematics, No. 52. Springer-Verlag, New York-Heidelberg, 1977. xvi+496 pp.


 

\bibitem{salamonviac}
Salamon S, Viaclovsky J.
Orthogonal complex structure on domains in $\mathbb{R}^{4}$.
Math. Ann. 343, 853--899 (2009)

\bibitem{segre}
Segre B.
The maximum number of lines lying on a quartic surface.
Quart. J. Math., Oxford Ser. 14, (1943). 86--96. 

\bibitem{shapiro}
Shapiro G. 
On discrete differential geometry in twistor space. 
Journal of Geometry and Physics, Volume 68, June 2013, Pages 81--102

\bibitem{stoppato}
Stoppato C. A new series expansion for slice regular functions. Adv. Math. 231 (2012), no. 3-4, 1401--1416. 

\bibitem{stoppatopoles}
Stoppato C.
Poles of regular quaternionic functions. 
Complex Var. Elliptic Equ. 54 (2009), no. 11, 1001--1018. 





\bibitem{wood}
Wood JC.
Harmonic morphisms and Hermitian structures on Einstein 4-manifolds.
Internat. J. Math. 3 (1992), no. 3, 415--439. 


\bibitem{xu}
Xu G.
Subvarieties of general hypersurfaces in projective space.
J. Differential Geom. 39 (1994), no. 1, 139--172. 

\end{thebibliography}
\end{document}